\documentclass[11pt,a4paper]{amsart}

\usepackage[utf8]{inputenc}
\usepackage[T1]{fontenc}

\usepackage{lmodern}

\usepackage{amsmath,amstext,amssymb,amsopn,amsthm,color,graphicx,mathrsfs,mathtools,dsfont}

\usepackage[margin=1in]{geometry}
\usepackage{float}
\usepackage[colorlinks,citecolor=blue,urlcolor=blue,bookmarks=true]{hyperref}

\usepackage{enumitem}
\usepackage{color}

\newtheorem{theorem}{Theorem}[section]
\newtheorem{lemma}[theorem]{Lemma}

\newtheorem{corollary}[theorem]{Corollary}

\theoremstyle{definition}

\newtheorem{example}[theorem]{Example}
\theoremstyle{remark}
\newtheorem{remark}[theorem]{Remark}

\newcommand{\E}{\mathbb{E}}

\newcommand{\N}{\mathbb{N}}
\newcommand{\Sd}{\mathbb{S}^{d-1}}
\renewcommand{\P}{\mathbb{P}}
\newcommand{\R}{\mathbb{R}}
\newcommand{\Rd}{{\mathbb{R}^d}}

\newcommand{\rf}[1]{\eqref{#1}}
\newcommand{\ve}{{\varepsilon }}

\newcommand{\bfX}{\mathbf{X}}

\newcommand{\calF}{\mathcal{F}}

\DeclareMathOperator{\dist}{dist}

\DeclareMathOperator{\supp}{supp}
\DeclareMathOperator*{\esssup}{ess\,sup}
\newcommand{\ind}{\mathds{1}}

\newcommand{\one}{{\bf 1}}

\renewcommand{\subset}{\subseteq}

\newcommand{\vphi}{\varphi}
\renewcommand{\epsilon}{\varepsilon}

\renewcommand{\leq}{\leqslant}
\renewcommand{\le}{\leq}
\renewcommand{\geq}{\geqslant}
\renewcommand{\ge}{\geq}

\newcommand{\od}{{\textrm d}}
\newcommand{\ud}{\,{\textrm d}}
\newcommand{\indyk}{\mathbf{1}}

\definecolor{kb}{rgb}{0.1,0.5,0.1}

\numberwithin{equation}{section}
\usepackage{amsaddr}

\begin{document}
	
	\title{Self-similar solution for fractional Laplacian in cones}
	
	\author[K. Bogdan, P. Knosalla, \L{}. Le\.{z}aj and D. Pilarczyk]{Krzysztof Bogdan$^{1,*}$, Piotr Knosalla$^2$, \L{}ukasz Le\.{z}aj$^{1,3,\dagger}$ and Dominika Pilarczyk$^{1,\ddagger}$}
	\address{$^1$Wroc\l{}aw University of Science and Technology, Faculty of Pure and Applied Mathematics, wyb. Wyspia\'{n}skiego 27, 50-370 Wroc\l{}aw, Poland \\ $^*$\href{mailto:krzysztof.bogdan@pwr.edu.pl}{krzysztof.bogdan@pwr.edu.pl}, $^\dagger$\href{lukasz.lezaj@pwr.edu.pl}{lukasz.lezaj@pwr.edu.pl}, $^\ddagger$\href{dominika.pilarczyk@pwr.edu.pl}{dominika.pilarczyk@pwr.edu.pl}}
	\address{$^2$University of Opole, Institute of Physics, ul. Oleska 48, 45-052 Opole, Poland \\ \href{piotr.knosalla@uni.opole.pl}{piotr.knosalla@uni.opole.pl}}
	\address{$^3$Department of Mathematics and Statistics, University of Jyv\"{a}skyl\"{a}, P.O. Box 35, 40014 Jyv\"{a}skyl\"{a}, Finland}

	\thanks{Krzysztof Bogdan was partially supported by the National Science Centre (Poland):
	grant 2017/27/B/ST1/01339. \L{}ukasz Le\.{z}aj was partially supported by the National Science Centre (Poland): grant 2021/41/N/ST1/04139.}
	
	\date{\today}
	
	\begin{abstract}
We construct a self-similar solution of the heat equation
for the fractional Laplacian with Dirichlet boundary conditions
in every fat cone. Furthermore, we give the entrance law from the vertex and the Yaglom limit for the corresponding killed isotropic stable L\'evy process and precise large-time asymptotics for solutions of the Cauchy problem in the cone. 
	\end{abstract}
	
	\maketitle
	
	%
	%
	\section{Introduction}

	Let $d\in \N:=\{1,2,\ldots\}$. Consider  an arbitrary non-empty open cone $\Gamma\subset \Rd$. Thus, $ry \in \Gamma$ whenever $y \in \Gamma$ and $r>0$. 
Let $p_t^{\Gamma}(x,y)$, $t>0$, $x,y\in \Gamma$, be the Dirichlet heat kernel of the cone for the fractional Laplacian. More specifically,  for $\alpha\in (0,2)$,
$p^\Gamma$ is the transition density of the isotropic $\alpha$-stable L\'evy process in $\Rd$ killed upon leaving $\Gamma$; see, e.g., Bogdan and Grzywny \cite {KBTG10}. Let $M_{\Gamma}\colon \Rd\to [0,\infty)$ be the Martin kernel of $\Gamma$ with the pole at infinity (for definitions, see Section~\ref{sec:prelims}). The function is homogeneous (or self-similar) of some degree $\beta\in [0,\alpha)$. 
Our first result captures the asymptotics of $p_t^{\Gamma}$ at the vertex $0$ of $\Gamma$, as follows.
	\begin{theorem}\label{thm:main1} If the cone $\Gamma$ is fat, then for $s,t >0$, $x \in \Gamma$, we have
	\begin{equation*}
		\Psi_t(x):=\lim\limits_{\Gamma \in y \to 0} p_t^{\Gamma}(x,y)/M_{\Gamma}(y)\in (0,\infty),
		\end{equation*} 
		\begin{equation}\label{eq:18}
			\Psi_t(x) = t^{-(d+\beta)/\alpha}\Psi_1\big(t^{-1/\alpha}x\big), 
		\end{equation}
		and
		\begin{equation}\label{eq:psi_evol}
			\int_\Gamma \Psi_s(y)p^\Gamma_t(x,y)\ud y=\Psi_{s+t}(x).
		\end{equation}
	\end{theorem}	
The proof of Theorem \ref{thm:main1} is given in Section \ref{sec:doob}. 
In view of \eqref{eq:18} and \eqref{eq:psi_evol}, $\Psi_t(x)$ 
is a self-similar semigroup solution of the heat equation for the fractional Laplacian with Dirichlet conditions. By
\eqref{eq:psi_evol},
$\Psi_t$ is also an entrance law for $p^{\Gamma}$ at the origin, see, e.g., Blumenthal \cite[p. 104]{RB92}, Haas and Rivero \cite{HassRivero12} or Ba\~{n}uelos et al. \cite{MR1042064}.

The result is the next step in the program 
for the 
boundary potential theory 
sketched in the Introduction of Bogdan at al. \cite{KBTGMR10}, building on the boundary Harnack principle, Green function, and Dirichlet heat kernel estimates. 
In Theorems \ref{thm:YL} and \ref{thm:UYL} below, we make further applications to Probability by obtaining the so-called Yaglom limit for $\Gamma$. We also give applications to Functional Analysis and Partial Differential Equations, as follows.
 
Let $1 \leq q \leq \infty$ and $L^q(\Gamma) := L^q(\Gamma,\od x)$. For a weight function $w>0$, we denote $L^q(w):= L^q(\Gamma, w(x) \ud x)$. For instance, $L^1(M_{\Gamma}) = \{f/M_{\Gamma} \colon f \in L^1\}$. Then, for $1\le q <\infty$, we define
	\begin{equation*}
		\|f\|_{q,M_\Gamma }:=\|f/M_\Gamma\|_{L^q(M_\Gamma^2)}
		=\bigg( \int_{\Gamma} |f(x)|^q M_\Gamma^{2-q}(x)\ud x \bigg) ^{\frac{1}{q}}
		=\|f\|_{L^q(M_\Gamma^{2-q})} ,
	\end{equation*}      
	and, for $q=\infty$, we let
	\begin{equation*}
		\|f\|_{\infty, M_\Gamma}:=\esssup_{x\in \Gamma} |f(x)|/M_\Gamma(x).
	\end{equation*}
	Of course, $\|f\|_{1,M_\Gamma}=\|f\|_{L^1(M_\Gamma)}$. For a non-negative or integrable function $f$ we let
	\[
	P_t^{\Gamma}f(x):= \int_{\Gamma} p_t^{\Gamma}(x,y)f(y)\ud y, \quad t>0,\,x \in \Gamma.
	\]
We say that the cone $\Gamma$ is \emph{smooth} if its boundary is $C^{1,1}$ outside of origin, to wit, there is $r>0$ such that at every boundary point of  $\Gamma$ on the unit sphere $\Sd$, there exist inner and outer tangent balls for $\Gamma$, with radii $r$. Put differently, for $d\ge 2$, the \textit{spherical cap}, $\Gamma\cap \Sd$, is a $C^{1,1}$ subset of~$\Sd$. For instance, the right-circular cones (see Section \ref{sec:prelims}) are smooth.

The second result describes the large-time asymptotic behavior of the semigroup $P_t^{\Gamma}$.
	\begin{theorem}\label{thm:main2}
		Let $q \in [1,\infty)$. Assume that the cone $\Gamma$ is smooth with $\beta \geq \alpha/2$. Then for every $f \in L^1(M_{\Gamma})$ and $A=\int_{\Gamma}f(x)M_{\Gamma}(x)\,\mathrm{d} x$ we have
		\begin{equation}\label{e.sc}
		\lim_{t \to \infty} t^{\frac{d+2\beta}{\alpha }\frac{q-1}{q}}\|P_t^{\Gamma}f-A\Psi_t\|_{q,M_{\Gamma}}=0.
		\end{equation}
	\end{theorem}
	As stated in Lemma~\ref{lem:rc_cones}, the condition $\beta\ge \alpha/2$ is sharp for smooth cones and $d\ge 2$; see also Example~\ref{ex:rc_cones}. Theorem~\ref{thm:main2} follows from the more general Theorem \ref{thm_lim_norm}, by means of 
	Corollary \ref{cor:s_cones}.

	Let us comment on previous developments in the literature and our methods.
	 If $\Gamma=\Rd$, then $\beta=0$, $M_\Gamma=1$ and $p^\Gamma_t (x,y)=p_t(x,y)$ is the transition density of the fractional Laplacian on $\Rd$ (see below). In this case, $\Psi_t(x)=p_t(0,y)$ and Theorem \ref{thm:main2} was resolved by V\'{a}zquez \cite{Vasquez18}, see also Bogdan et al. \cite{BJKP22} with $\kappa=0$ in \cite[Eq. (1.1)]{BJKP22}; see  Example \ref{ex:beta0} below, too. 
	 
	 For general cones $\Gamma$, the behavior of $p_t^\Gamma$ is intrinsically connected to properties of $M_{\Gamma}$, see, e.g., \cite{KBTG10}, \cite{BoPaWa2018}, or Kyprianou et al. \cite{MR4415390}. The identification of  the Martin kernel $M_\Gamma$ was accomplished by Ba\~{n}uelos and Bogdan \cite{BB04}. Its crucial property is the homogeneity of order $\beta \in [0,\alpha)$, which is also reflected in the behavior of the Green function studied by Kulczycki \cite{Kulczycki99} and Michalik \cite[Lemma 3.3]{Michalik06}, at least when $\Gamma$ is a right-circular cone. As we see in Theorems \ref{thm:main1} and \ref{thm:main2}, the exponent $\beta$ determines the self-similarity of the semigroup solution and the asymptotic behavior of the semigroup $P_t^{\Gamma}$, too. For more information on $\beta$ we refer to \cite{BB04} and Bogdan et al. \cite{MR3339224}.

If $\Gamma$ is a Lipschitz cone, then Theorem \ref{thm:main1} follows from \cite[Corollary 3.2 and Theorem 3.3]{BoPaWa2018}. However, the method presented in \cite{BoPaWa2018} does not apply to general fat cones, in particular, to $\Gamma=\R \setminus \{0\}$ or $\Gamma = \R^2 \setminus ([0,\infty) \times \{0\})$, which are of interest for $\alpha \in (1,2)$. 

Instead, in this work we develop an approach suggested by \cite{BJKP22}, where the authors employ a stationary density of an Ornstein-Uhlenbeck type semigroup corresponding to a homogeneous (self-similar) heat kernel in a different setting (on the whole of $\Rd$). Another key tool in their analysis is the so-called Doob conditioning using an invariant function or the heat kernel.
Due to Theorem \ref{thm:M_invariance} below, the Martin kernel $M_\Gamma$ is invariant with respect to $P_t^{\Gamma}$, which indeed allows for Doob conditioning. Then we form the corresponding Ornstein-Uhlenbeck semigroup and prove existence of a stationary density $\vphi$ in Theorem \ref{thm:stat_density} by using the Schauder-Tychonoff fixed-point theorem. As we shall see in the proof of Theorem \ref{thm:main1}, the self-similar semigroup solution $\Psi_t$ is directly expressed by $\vphi$ and $M_{\Gamma}$. 

The remaining developments in our paper are as follows. In Subsection \ref{subsec:Y} we obtain an asymptotic relation between the Martin kernel and the survival probability near the vertex of the cone (see Corollary \ref{cor:survM_as}). We also obtain a Yaglom limit (quasi-stationary distribution) in Theorem \ref{thm:YL}, which describes the behavior of the stable process starting from a fixed point $x \in \Gamma$ and conditioned to stay in a cone, generalizing Theorem 1.1 of \cite{BoPaWa2018}. In Theorem \ref{thm:UYL} we considerably extend both results by allowing arbitrary initial distributions with finite moment of order $\alpha$.

We note that the Yaglom limit for random walks in cones is discussed by Denisov and Wachtel \cite{MR4155177},
but generally there are very few results in literature on unbounded sets. For a broad survey on quasi-stationary distributions, we refer to van Doorn and Pollet \cite{MR3063313}; see also Champagnat and Villemonais \cite{MR4546021}. 
Self-similar solutions for general homogeneous semigroups are discussed in Cholewa and Rodriguez-Bernal \cite{MR4596024}. Patie and Savov \cite{MR4320772} discuss generalized Ornstein-Uhlenbeck semigroups, which they call generalized Laguerre semigroups. 
Results related to Theorem \ref{thm:main2}, but for fractal Burgers equation and fractional $p$-Laplacian can be found in Biler et al. \cite[Theorem 2.2]{BILER2001613} and V\'{a}zquez \cite[Theorem 1.2]{VAZQUEZ2020112034}, respectively. 
For an approach to entrance laws based on fluctuation theory of Markov additive processes, we refer to \cite{MR4155177}, see also Chaumont et al. \cite{MR3160562}.

This brief review shows that the asymptotics of the heat kernel is in a busy intersection between Probability, Partial Differential Equations and Functional Analysis even though the fields do not always communicate. Our approach should apply to rather general self-similar operator semigroup kernels, at least when they enjoy an invariant function and suitable upper and lower bounds.

Here is a simple example to illustrate our findings.
\begin{example}\label{ex.pp}
If $d=1$, $\alpha\in (0,2)$, and $\Gamma=(0,\infty)$, then $\beta=\alpha/2$, $M_\Gamma(x)=(0\vee x)^{\alpha/2}$ for $x\in \R$, so, by Theorem~\ref{thm_lim_norm}, $\lim_{t \to \infty} t^{(1+\alpha)/(2\alpha)}\|P_t^{\Gamma}f\|_{2}=0$ if $\int_0^\infty f(x)x^{\alpha/2}\ud x=0$; see Example~\ref{ex:1}. Furthermore, by Remark~\ref{M_P_comp} and \eqref{eq:21},
the survival probability is $\P_x(\tau_\Gamma>t)\approx  (1 \wedge t^{-1/\alpha}x)^{\alpha/2}$ and the heat kernel satisfies $p_1^{\Gamma}(x,y) \approx (1+|x-y|)^{-1-\alpha} (1 \wedge x)^{\alpha/2} (1 \wedge y)^{\alpha/2}$, so $\Psi_t(x)\approx (t^{1/\alpha} \vee x)^{-1-\alpha}(t^{1/\alpha}\wedge x)^{\alpha/2}$. Here and below $x,y,t>0$. By Lemma~\ref{lem:vphi_redef}, the stationary density of the corresponding Ornstein-Uhlenbeck semigroup is $\varphi(x)\approx (1+x)^{-1-3\alpha/2}$ and the Yaglom limit has density $\varphi(x)x^{\alpha/2}/\int_0^\infty \varphi(y)y^{\alpha/2}\ud y \approx x^{\alpha/2}(1+x)^{-1-3\alpha/2} \approx \Psi_1(x)$; see Theorem~\ref{thm:YL}.  We refer to \cite[Example 5]{HassRivero12} for an exact but less explicit expression for the Yaglom limit by means of exponential functionals. See also Example~\ref{ex:1} for the case of $\Gamma = \R \setminus \{0\}$ and Example~\ref{ex:rc_cones} for $d$-dimensional extensions.
\end{example}

\subsection*{Acknowledgments} 
Part of the research for this work was conducted during \L{}ukasz Le\.{z}aj's post-doctoral stay at the University of Jyv\"{a}skyl\"{a} from January to June 2022. He expresses his gratitude to the University for their hospitality and to Professor Stefan Geiss for his warm welcome, guidance, and support during the stay.

	\section{Preliminaries}\label{sec:prelims}
	For $x,z \in \Rd$, the standard scalar product of is denoted by $x \cdot z$ and $|z|$ is the Euclidean norm. For $x \in \Rd$ and $r\in (0,\infty)$, we let $B(x,r) = \{y \in \Rd \colon |x-y|<r\}$, the ball centered at $x$ with radius $r$, and write $B_r:=B(0,r)$. All the considered sets, functions and measures are Borel. For non-negative functions $f,g$, we write $f \approx g$ if there is a number $c\in (0,\infty)$, i.e., a \textit{constant}, such that $c^{-1}f\le g\le c f$, and write $f \lesssim g$ if there is a constant $c$ such that $f \le cg$. As usual, for two real numbers $a,b \in \R$, we write $a \wedge b := \min \{a,b\}$ and $a \vee b := \max \{a,b\}$.
	
	Recall that $\alpha \in (0,2)$ and  let
	\[
	\nu(z) = c_{d,\alpha} |z|^{-d-\alpha}, \quad z \in \Rd,
	\]
	where the constant $c_{d,\alpha}$ is such that
	\[
	\int_{\Rd} \big( 1-\cos (\xi \cdot z) \big) \nu(z) \ud z = |\xi|^{\alpha}, \quad \xi \in \Rd.
	\]
	For $t>0$ we let
	\begin{equation}\label{eq:inv_F_tr}
		p_t(x) := (2\pi)^{-d}\int_{\Rd} e^{-t|\xi|^{\alpha}} e^{-i\xi \cdot x}\ud \xi, \quad x \in \Rd.
	\end{equation}
	By the L\'{e}vy-Khintchine formula, $p_t$ is a probability density function and 
	\begin{equation*}
		\int_{\Rd} e^{i\xi \cdot x}\,p_t(x)\ud x 
		= e^{-t|\xi|^{\alpha}}, \quad \xi \in \Rd, \;t >0.
	\end{equation*}
	We consider the isotropic $\alpha$-stable L\'evy process $\bfX=(X_t,t\ge 0)$ in $\Rd$, with $$p_t(x,y):=p_t(y-x), \quad x,y \in \Rd,\ t>0,$$  
	as transition density. Thus,
	\begin{equation*}
		\E_x e^{i\xi \cdot X_t} = 
		\int_{\Rd} e^{i\xi \cdot y}\,p_t(x,y) \ud y 
		= e^{i\xi\cdot x-t|\xi|^{\alpha}}, \quad \xi \in \Rd,\, x\in \Rd, \;t >0.
	\end{equation*}
	%
	The L\'{e}vy-Khintchine exponent of $\bfX$ is, of course, $|\xi|^\alpha$ and $\nu$ is the intensity of jumps. 
	%
	By \eqref{eq:inv_F_tr},
	\begin{align}
		\label{eq:dhk_approx_scaling}
		p_t(x,y) &= t^{-d/\alpha}p_1 \big(t^{-1/\alpha}x,t^{-1/\alpha}y\big), \quad x,y \in \Rd, \;t>0,
		\intertext{and}
		\label{eq:iso}
		p_t\left( Tx,Ty\right)& = p_t(x,y),\quad x,y \in \Rd, \;t>0,
	\end{align}
	for every isometry $T$ on $\Rd$.
	It is well known that
	\begin{equation}\label{eq:hk}
		p_t(x,y) \approx t^{-d/\alpha} \wedge t|y-x|^{-d-\alpha}, \quad x,y \in \Rd, \;t>0,
	\end{equation}
	see, e.g., \cite{MR0119247}. 
	We then consider the time of the first exit of $\bfX$ from the cone $\Gamma$, 
	\[
	\tau_{\Gamma} := \inf \{t \geq 0 \colon X_t \notin \Gamma\},
	\]
and we define 
the Dirichlet heat kernel for $\Gamma$,
	\begin{equation*}
	p_t^{\Gamma}(x,y) := p_t(x,y) -\E_x \left[ \tau_{\Gamma}<t;p_{t-\tau_{\Gamma}} \left(  X_{\tau_{\Gamma}},y\right) \right], \quad x,y \in \Gamma, \;t>0,
	\end{equation*}
	see \cite{KBTG10,bogdan2023shotdown, MR1329992}.
	It immediately follows that $p_t^{\Gamma}(x,y) \leq p_t(x,y)$ for all $x,y \in \Gamma$ and $t>0$. The Dirichlet heat kernel is non-negative, and symmetric: $p_t^{\Gamma}(x,y)=p_t^{\Gamma}(y,x)$ for $x,y \in \Gamma$, $t>0$. It satisfies the Chapman-Kolmogorov equations:
	\begin{equation}\label{eq:9}
		p_{t+s}^{\Gamma}(x,y) = \int_{\Gamma}p_t^{\Gamma}(x,z)p_s^{\Gamma}(z,y)\ud z, \quad x,y \in \Gamma, \;s, t>0.
	\end{equation}
	For nonnegative or integrable functions $f$ we define the \textit{killed semigroup} 
	by
	\[
	P_t^{\Gamma} f(x) := \E_x \left[ \tau_{\Gamma}>t; f(X_t) \right] = \int_{\Gamma} p_t^{\Gamma}(x,y)f(y)\ud y, \quad x \in \Gamma,\;t>0.
	\]
	In particular, for $f \equiv 1$ we obtain 
	the \textit{survival probability}:
	\begin{equation}\label{eq:surv_def}
		\P_x(\tau_{\Gamma}>t) = \int_{\Gamma} p_t^{\Gamma}(x,y)\ud y, \quad x \in \Gamma, \,t>0,
	\end{equation}
	see \cite[Remark 1.9]{KBTGMR-dhk}.
	Since $t^{-1/\alpha}\Gamma=\Gamma$, the scaling \eqref{eq:dhk_approx_scaling} extends to the Dirichlet heat kernel:
	\begin{equation*}
		p_t^{\Gamma}(x,y) = t^{-d/\alpha}p_1^{\Gamma} \big( t^{-1/\alpha}x,t^{-1/\alpha}y \big), \quad x,y \in \Gamma,\; t>0.
	\end{equation*}
	As a consequence,
	\begin{equation}\label{eq:surv_scaling}
		\P_x (\tau_{\Gamma}>t) = \P_{t^{-1/\alpha}x}(\tau_{\Gamma}>1), \quad x \in \Gamma,\;t>0.
	\end{equation}
	Furthermore, by \eqref{eq:iso},
	\begin{equation}\label{eq:gamma_T}
		p_t^{T\Gamma}(Tx,Ty) = p_t^\Gamma (x,y), \quad x,y \in \Gamma,\; t>0.
	\end{equation}
	The operators $P_t^\Gamma$ and the kernel $p_t^\Gamma (x,y)$ are the main subject of the paper. In view of \rf{eq:gamma_T}, without loss of generality we may assume that $\one := (0,\ldots,0,1) \in \Gamma$. 
By \cite[Theorem 3.2]{BB04}, there is a unique non-negative function $M_{\Gamma}$ on $\Rd$ such that $M_{\Gamma}(\one)=1$, $M_{\Gamma}=0$ on $\Gamma^c$, and for every open bounded set $B \subset \Gamma$,
	\begin{equation}\label{Mgamma_E}
		M_{\Gamma}(x) = \E_x M_{\Gamma}(X_{\tau_B}), \quad x \in \Rd.
	\end{equation}
	Moreover, $M_{\Gamma}$ is locally bounded on $\Rd$ and homogeneous of some order $\beta \in [0,\alpha)$, i.e.,
	\begin{equation}\label{eq:M_hom}
		M_{\Gamma}(x) = |x|^{\beta}M_{\Gamma}(x/|x|), \quad x \in \Gamma.
	\end{equation}
	We call $M_\Gamma$ the Martin kernel of $\Gamma$ with the pole at infinity.
\begin{example}
By \cite{BB04}, $\beta=\alpha/2$ if $\Gamma$ is a half-space and $\beta=\alpha-1$ if $\Gamma=\R\setminus \{0\}$ and $1<\alpha<2$.
By \cite{MR2182071}, $\beta=(\alpha-1)/2$ if $\Gamma=\R^2 \setminus ([0,\infty) \times \{0\})$ and $1<\alpha<2$.
\end{example}

	Below we often assume that $\Gamma$ is \emph{fat}, i.e., $\kappa \in (0,1)$ exists such that for all $Q \in \overline{\Gamma}$ and $r \in (0,\infty)$, there is a point $A = A_r(Q) \in \Gamma \cap B(Q,r)$ such that $B(A,\kappa r) \subseteq \Gamma \cap B(Q,r)$, see \cite[Definition 1]{KBTGMR10}. 
Recall that $\Gamma$ is \emph{smooth} if  $d=1$ or $d\ge 2$ and $\Gamma \cap \Sd$ is a $C^{1,1}$ subset of $\Sd$.
Furthermore, a cone $\Gamma$ is called \emph{right-circular}, if $\Gamma = \{x=(x_1,\ldots,x_d) \in \Rd \setminus \{0\}: x_d>|x| \cos \eta\}$, with $\eta \in (0,\pi)$ called the angle of the cone. Of course, every right-circular cone is smooth, and every smooth cone is fat.
	
	By \cite[Theorem 1]{KBTGMR10}, the following approximate factorization holds true for fat cones:
	\begin{equation}\label{eq:dhk_approx}
		p_t^{\Gamma}(x,y) \approx  \P_x(\tau_{\Gamma}>t) p_t(x,y) \P_y(\tau_{\Gamma}>t), \quad x,y \in \Gamma,\; t>0.
	\end{equation}
For $R\in (0,\infty)$, we let $\Gamma_R := \Gamma \cap B_R$, the \textit{truncated cone}. 
	
	\section{Doob conditioning}\label{sec:doob}
	The Martin kernel $M_{\Gamma}$ is invariant for the semigroup $P_t^{\Gamma}$, as follows.
	\begin{theorem}\label{thm:M_invariance}
		For all $x \in \Gamma$ and $t>0$, we have $P_t^{\Gamma} M_{\Gamma}(x) = M_{\Gamma}(x)$.
	\end{theorem}
	\begin{proof}
		Fix $t>0$ and $x\in\Gamma$. We have
		\begin{equation}\label{eq:2}
			P_t^{\Gamma} M_{\Gamma}(x) = \E_x \Big[\tau_{\Gamma}>t;\, M_{\Gamma}(X_t) \Big].
		\end{equation}
		Let $R>0$ and $\tau_R:=\tau_{\Gamma_R}$. By \eqref{Mgamma_E} and the strong Markov property,
		\begin{equation}\label{eq:3}
			M_{\Gamma}(x) = \E_x M_{\Gamma}(X_{\tau_{\Gamma_R}})= \E_x  M_{\Gamma} \big( X_{t \wedge \tau_R} \big) = \E_x  \Big[ X_{t \wedge \tau_R} \in \Gamma;\,M_{\Gamma} \big( X_{t \wedge \tau_R} \big) \Big],
		\end{equation}
		where the last equality follows from the fact that $M_{\Gamma} = 0$ outside $\Gamma$. We note that $\P_x-a.s.$, $\tau_R \to \tau_{\Gamma}$ as $R \to \infty$ (see, e.g., \cite[proof of Proposition A.1]{ABGLW21}). We consider two scenarios. On $\{\tau_{\Gamma}=\infty\}$, for $R$ large enough, we have: $\tau_R > t$, $\ind_{X_{t \wedge \tau_R} \in \Gamma} = 1 = \ind_{t<\tau_{\Gamma}}$, and
			\[
			M_{\Gamma} \big( X_{t \wedge \tau_R} \big) \ind_{X_{t \wedge \tau_R} \in \Gamma} = M_{\Gamma}(X_t) = M_{\Gamma}(X_t)\ind_{t<\tau_{\Gamma}}.
			\]
			On $\{\tau_{\Gamma}<\infty\}$, for $R$ large enough we have: $\tau_R = \tau_{\Gamma}$, $\ind_{X_{t \wedge \tau_R} \in \Gamma} = \ind_{t<\tau_{\Gamma}}$, and
			\[
			M_{\Gamma} \big( X_{t \wedge \tau_R} \big) \ind_{X_{t \wedge \tau_R} \in \Gamma} = M_{\Gamma}(X_t) \ind_{t<\tau_{\Gamma}},
			\]
			too. In both cases, the integrand on the right-hand side of \eqref{eq:3} converges $a.s.$ to the integrand on the right-hand side of \eqref{eq:2} as $R \to \infty$. By the local boundedness of $M_{\Gamma}$ and \eqref{eq:M_hom},
		\[
		\Big| M_{\Gamma} \big( X_{t \wedge \tau_R} \big) \ind_{X_{t \wedge \tau_R} \in \Gamma} \Big| \leq c  \big| X_{t \wedge \tau_R} \big|^{\beta} \leq c (X_t^*)^{\beta},
		\]
		where
		\[
		X_t^* := \sup_{0 \leq s \leq t} |X_s|.
		\]
		Using \cite[Theorem 2.1]{MR4490672} and the fact that $\beta\in[0,\alpha$), we conclude that $\E_x(X_t^*)^{\beta}<\infty$. An application of the dominated convergence theorem ends the proof.
	\end{proof}
	\subsection{Renormalized kernel}
	We define the renormalized (Doob-conditioned) kernel
	\begin{equation}\label{eq:rho_def}
		\rho_t(x,y) = \frac{p_t^{\Gamma}(x,y)}{M_{\Gamma}(x)M_{\Gamma}(y)}, \quad x,y \in \Gamma,\;t>0.
	\end{equation}
	Note that $\rho$ is jointly continuous. By Theorem \ref{thm:M_invariance},
	\begin{equation}\label{eq:rho_density}
		\int_{\Gamma} \rho_t(x,y)M_{\Gamma}^2(y)\ud y=1, \quad x \in \Gamma,\; t>0,
	\end{equation}
	and by \eqref{eq:9},
	\begin{equation}\label{eq:rho_ChK}
	\int_{\Gamma} \rho_t(x,y)\rho_s(y,z) M_{\Gamma}^2(y)\ud y = \rho_{t+s} (x,z), \quad x,y \in \Gamma,\;s,t>0.
	\end{equation}
	In other words, $\rho_t$ is a symmetric transition probability density on $\Gamma$ with respect to the measure $M^2_{\Gamma}(y)\ud y$. Furthermore, the following scaling property holds true: for all $x,y \in \Gamma$ and all $t>0$,
	\begin{equation}\label{eq:scaling1}
		\rho_t(x,y) = \frac{t^{-d/\alpha}p_1^{\Gamma}(t^{-1/\alpha}x,t^{-1/\alpha}y)}{t^{2\beta/\alpha}M_{\Gamma}(t^{-1/\alpha}x)M_{\Gamma}(t^{-1/\alpha}y)} = t^{-(d+2\beta)/\alpha} \rho_1 (t^{-1/\alpha}x,t^{-1/\alpha}y).
	\end{equation}
	Therefore,
	\begin{equation}\label{eq:scaling2}
		\rho_{st}(t^{1/\alpha}x,t^{1/\alpha}y) = t^{-(d+2\beta)/\alpha} \rho_s(x,y), \quad x,y \in \Gamma,\;s,t>0.
	\end{equation}
	By \eqref{eq:dhk_approx}, for fat cones we have
	\begin{equation}\label{eq:rho_factor}
		\rho_t(x,y) \approx \frac{\P_x (\tau_{\Gamma}>t)}{M_{\Gamma}(x)}p_t(x,y)  \frac{\P_y (\tau_{\Gamma}>t)}{M_{\Gamma}(y)}, \quad x,y \in \Gamma,\, t>0.
	\end{equation}
	The boundary behavior of $\P_x(\tau_{\Gamma}>t)/M_\Gamma (x)$ is important due to \eqref{eq:rho_factor}, but it is rather elusive. 
	The next lemma strengthens the upper bound from \cite[Lemma 4.2]{BB04}. 
	\begin{lemma}\label{lem:sur_M_comp}
		There exists a constant $c$ depending only on $\alpha$ and $\Gamma$, such that
		\begin{equation}\label{eq:5}
			\P_x(\tau_{\Gamma}>t) \leq c\Big(t^{-\beta/\alpha}+t^{-1}|x|^{\alpha-\beta}\Big) M_{\Gamma}(x), \quad t>0,\,x \in \Gamma.
		\end{equation}
	\end{lemma}

	\begin{remark} (1) For $t=1$, \eqref{eq:5} reads as follows,
				\begin{equation}\label{M_P_comp}
				\P_x(\tau_{\Gamma}>1) \leq c(1+|x|^{\alpha-\beta})M_{\Gamma}(x), \quad x \in \Gamma.
			\end{equation}
(2) The estimate \eqref{eq:5} applies to arbitrary cones and arguments $t,x$, 
however, it is not optimal. For example, for the right-circular cones, we can confront \eqref{eq:rho_factor} with
		\[
		M_{\Gamma}(x) \approx \delta_{\Gamma}(x)^{\alpha/2}|x|^{\beta-\alpha/2}, \quad x \in \Gamma,
		\]
		and
		\[
		\P_x(\tau_{\Gamma}>1) \approx \big(1\wedge\delta_{\Gamma}(x)\big)^{\alpha/2}\big(1\wedge|x|\big)^{\beta-\alpha/2}, \quad x \in \Gamma,
		\]
as provided by \cite[Lemma 3.3]{Michalik06} and \cite[Example 7]{KBTGMR10}. \\
(3) For the right-circular cones, the ratio 
		\[
		\frac{\P_x(\tau_{\Gamma}>1)}{M_{\Gamma}(x)} \approx \frac{\big(1+\delta_{\Gamma}(x))^{-\alpha/2}}{(1+|x|)^{\beta-\alpha/2}}, \quad x \in \Gamma,
		\]
		is bounded if and only if $\beta \geq \alpha/2$.
	\end{remark}

	\begin{proof}[Proof of Lemma \ref{lem:sur_M_comp}]
		We slightly modify the proof of \cite[Lemma 4.2]{BB04}. First, suppose that $t=1$. The case $x \in \Gamma_1$ in \eqref{M_P_comp} is resolved by \cite[Lemma 4.2]{BB04}, so we assume that $x \in \Gamma \setminus \Gamma_1$.  For every $z \in \Rd \setminus \{0\}$ we define its projection on the unit sphere $\tilde{z} := z/|z|$. By \eqref{eq:surv_scaling},
		\[
		\P_x(\tau_{\Gamma}>1) = \P_{\tilde{x}}(\tau_{\Gamma}>|x|^{-\alpha}).
		\]
		Then we have
		\[
		\P_{\tilde{x}}(\tau_{\Gamma}>|x|^{-\alpha}) \leq \P_{\tilde{x}}(\tau_{\Gamma_2}>|x|^{-\alpha}) + \P_{\tilde{x}}(\tau_{\Gamma_2}<\tau_{\Gamma}).
		\]
		By the boundary Harnack principle (BHP), see Song and Wu \cite[Theorem 3.1]{MR1719233}, and the homogeneity of $M_{\Gamma}$ \eqref{eq:M_hom},
		\begin{equation}\label{eq:22}
			\P_{\tilde{x}}(\tau_{\Gamma_2}<\tau_{\Gamma}) \leq \P_{\one}(\tau_{\Gamma_2}<\tau_{\Gamma})M_{\Gamma}(\tilde{x}) = c_1 |x|^{-\beta}M_{\Gamma}(x).
		\end{equation}
		We let
		\[
		c_2 = \inf_{y \in \Gamma_2} \int_{\Gamma \setminus \Gamma_2} \nu(y-z)\ud z.
		\]
		Clearly, $c_2>0$. We recall the Ikeda-Watanabe formula: 
	\begin{align*}
		\P_x[\tau_D\in I,\; Y_{\tau_D-}\in A,\; Y_{\tau_D}\in B]=
		\int_I  \int_{A} \int_B  \, p_s^D(x,v)\nu(v,z)\ud z  \ud  v \ud  s,
	\end{align*}
	where $x\in D$, $I\subset [0,\infty)$, $A\subset D$ and $B\subset D^c$, see, e.g., Bogdan et al. \cite[Section~4.2]{MR3737628}. By Markov inequality and BHP,
		\begin{align*}
			\P_{\tilde{x}}\big(\tau_{\Gamma_2}>|x|^{-\alpha}\big) &\leq |x|^{\alpha} \E_{\tilde{x}}\tau_{\Gamma_2} = |x|^{\alpha} \int_{\Gamma_2}G_{\Gamma_2}(\tilde{x},y)\ud y \\ &\leq c_2^{-1} |x|^{\alpha} \int_{\Gamma \setminus \Gamma_2} \int_{\Gamma_2}G_{\Gamma_2}(\tilde{x},y) \nu(y-z) \ud y\ud z \\ &\leq c_2^{-1} |x|^{\alpha} \P_{\tilde{x}} \big(X_{\tau_{\Gamma_2}} \in \Gamma\big) \\ &\leq c_1c_2^{-1} |x|^{\alpha} \P_{\one} \big(X_{\tau_{\Gamma_2}} \in \Gamma\big) M_{\Gamma}(\tilde{x}) \\ &= c_1c_2^{-1}c |x|^{\alpha-\beta}M_{\Gamma}(x).
		\end{align*}
	By \eqref{eq:22}, we get \eqref{M_P_comp} when $x \in \Gamma \setminus \Gamma_1$. For arbitrary $t>0$, we use \eqref{eq:surv_scaling} and \eqref{M_P_comp}:
\begin{align*}
		\P_x(\tau_{\Gamma}>t) &= \P_{t^{-1/\alpha}x}(\tau_{\Gamma}>1) \\ &\leq c \left(1+\left(t^{-1/\alpha}|x|\right)^{\alpha-\beta}\right) M_{\Gamma}(t^{-1/\alpha}x)\\
		&= c\left(t^{-\beta/\alpha}+t^{-1}|x|^{\alpha-\beta}\right)M_{\Gamma}(x).
\end{align*}
	\end{proof}
	By the proof of \cite[Lemma 4.2]{BB04}, for every $R \in (0,\infty)$ there exists a constant $c$, depending only on $\alpha$, $\Gamma$ and $R$, such that 
	\begin{equation*}
		c^{-1}M_{\Gamma}(x)t^{-\beta/\alpha} \leq \P_x(\tau_{\Gamma}>t) \leq c M_{\Gamma}(x)t^{-\beta/\alpha}, \quad x \in \Gamma_{Rt^{1/\alpha}},\; t>0.
	\end{equation*}  
	In particular, for fat cones, in view of \eqref{eq:hk} and \rf{eq:rho_factor}, 
	\begin{equation}\label{eq:rho_bound}
		\rho_1(x,y) \approx (1+|y|)^{-d-\alpha} \frac{\P_y(\tau_{\Gamma}>1)}{M_{\Gamma}(y)}, \quad x\in \Gamma_R,\; y \in \Gamma,
	\end{equation}
	with comparability constant depending only on $\alpha$, $\Gamma$ and $R$.
 Using Lemma \ref{lem:sur_M_comp} we also conclude that for every $R \geq 1$ there is a constant $c$ depending only on $R$, $\alpha$ and $\Gamma$, such that
	\begin{equation}\label{eq:rho_bound2}
		\rho_1(x,y) \leq c(1+|y|)^{-d-\beta}, \quad x \in \Gamma_R, \;y \in \Gamma.
	\end{equation}

	\subsection{Ornstein-Uhlenbeck kernel}
	Encouraged by \cite{BJKP22}, we let
	\begin{equation}\label{eq:l_def}
		\ell_t(x,y) := \rho_{1-e^{-t}}(e^{-t/\alpha}x,y), \quad x,y \in \Gamma, \; t>0,
	\end{equation}
	and, by \eqref{eq:scaling2}, we get the Chapman-Kolmogorov property for $\ell_t$:
	\[
	\int_{\Gamma} \ell_t(x,y)\ell_s(y,z)M_{\Gamma}^2(y)\ud y = \ell_{t+s}(x,z), \quad x,z \in \Gamma, \; s,t>0.
	\]
	By \eqref{eq:rho_density},
	\begin{equation*}
		\int_{\Gamma} \ell_t(x,y)M_{\Gamma}^2(y)\ud y=1,\quad x \in \Gamma, \; t>0 .
		\end{equation*}
	Thus, $\ell_t$ is a transition probability density on $\Gamma$ with respect to $M_{\Gamma}^2(y)\ud y$.
	We define the corresponding Ornstein-Uhlenbeck semigroup:
	\[
	L_t f(y) = \int_{\Gamma} \ell_t(x,y) f(x)M_{\Gamma}^2(x)\ud x, \quad x\in \Gamma, t>0.
	\]
	We easily see that the operators are bounded on $L^1(M^2_\Gamma(y)\ud y)$. In fact, they preserve densities, i.e., functions $f\ge0$ such that $\int_\Gamma f(x)M^2_\Gamma(x)\ud x=1$.
	
	Before we immerse into details, let us note that the relations \eqref{eq:rho_bound} and \eqref{eq:rho_bound2} will be crucial in what follows. Both of them rely on the factorization of the Dirichlet heat kernel \eqref{eq:dhk_approx}, which is valid for fat sets. So, below in this section we assume (sometimes tacitly) that $\Gamma$ is a fat cone.
	\begin{theorem}\label{thm:stat_density}
		Assume $\Gamma$ is a fat cone. Then there is a unique stationary density $\vphi$ for the operators $L_t$, $t>0$.
	\end{theorem}
	\begin{proof}
		Fix $t>0$ and consider the family $F$ of non-negative functions on $\Gamma$ that have the form
		\[
		f(y) = \int_{\Gamma_1}\rho_t(x,y)\,\mu(\od x), \quad y \in \Gamma,
		\] 
		for some sub-probability measure $\mu$ concentrated on $\Gamma_1$. By \eqref{eq:rho_density}, $F \subseteq L^1(M_{\Gamma}^2(y))$. By the scaling \eqref{eq:scaling2} and the same reasoning as in the proof of \cite[Theorem 3.2]{BJKP22}, $L_t F \subseteq F$. Since $L_t$ is continuous, we also have $L_t \overline{F} \subseteq \overline{F}$, where $\overline{F}$ is the closure of $F$ in the norm topology of $L^1(\Gamma,M^2_\Gamma(y)\ud y)$. Next, we observe that $F$ is convex, therefore by \cite[Theorem 3.7]{Brezis11}, $\overline{F}$ is equal to the closure of $F$ in the weak topology. In view of \eqref{eq:rho_bound},
		%
		%
		\begin{equation}\label{eq:4}
			f(y) \lesssim (1+|y|)^{-d-\alpha} \frac{\P_y(\tau_{\Gamma}>1)}{M_{\Gamma}(y)}, \quad y \in \Gamma,
		\end{equation}
		uniformly for $f\in F$. Moreover, \eqref{eq:rho_density} and \eqref{eq:rho_bound} show that the right-hand side of \eqref{eq:4} is integrable with respect to $M_{\Gamma}^2(y)\ud y$. Therefore, the family $F$ is uniformly integrable with respect to $M^2_\Gamma(y)\ud y$. By \cite[Theorem 4.7.20]{Bogachev07}, $F$ is weakly pre-compact in $L^1(M^2_\Gamma(y))$, so $\overline{F}$ is weakly compact.
		%
		Furthermore, we invoke \cite[Theorem 3.10]{Brezis11} to conclude that $L_t$ is weakly continuous. By the Schauder-Tychonoff fixed point theorem \cite[Theorem 5.28]{Rudin91},
		there is a density $\vphi \in \overline{F}$ satisfying $L_t \vphi = \vphi$. It is unique by the strict positivity of the kernel $\ell_t$, and the same for every $t>0$, see the proof of \cite[Theorem 3.2]{BJKP22}.
	\end{proof}
	Let us note that by Theorem \ref{thm:stat_density} and \cite[Theorem 1 and Remark 2]{KulikScheutzow15}, the following stability result for kernels $l_t$ in $L^1(M_{\Gamma}^2(y)\ud y)$ holds true for every $x\in \Gamma$:
	\begin{equation}\label{eq:14}
		\int_{\Gamma} \big| \ell_t(x,y)-\vphi(y) \big| M_{\Gamma}^2(y)\ud y \to 0 \quad \text{as} \quad t \to \infty.
	\end{equation}
	We claim that the convergence in \eqref{eq:14} is in fact uniform for $x$ in any bounded subset $A \subseteq \Gamma$. Indeed, let $x,x_0 \in A$. In view of \eqref{eq:l_def} and \eqref{eq:rho_factor} we may write
	\begin{align}
		\nonumber
		\int_{\Gamma} \big| \ell_{1+t}(x,y) - \vphi(y) \big| M_{\Gamma}^2(y)\ud y &= \int_{\Gamma} \bigg| \int_{\Gamma}\ell_1(x,z) \big( \ell_t(z,y) - \vphi(y) \big) M_{\Gamma}^2(z)\ud z \bigg| M_{\Gamma}^2(y)\ud y \\ 
		\label{eq:15}
		&\leq c \int_{\Gamma} \ell_1(x_0,z) \int_{\Gamma} \big|  \ell_t(z,y) - \vphi(y)\big| M_{\Gamma}^2(y)\ud y\, M_{\Gamma}^2(z)\ud z.
	\end{align}
By \eqref{eq:14}, for every $z\in \Gamma$, 
	\[
	I_t(z):=\int_{\Gamma} \big| \ell_t(z,y) - \vphi(y) \big| M_{\Gamma}^2(y)\ud y \to 0 \quad \text{as} \quad t \to \infty.
	\]
	Moreover, $I_t(z) \leq \int_{\Gamma} \big( \ell_t(z,y) + \vphi(y) \big) M_{\Gamma}^2(y)\ud y = 2$. Since
	\[
	\int_{\Gamma} 2\ell_1(x_0,z)M_{\Gamma}^2(z)\ud z = 2 < \infty,
	\]
	by the dominated convergence theorem the iterated integral in \eqref{eq:15} tends to $0$ as $t \to \infty$, so the convergence in \eqref{eq:14} is uniform for all $x \in A$, as claimed. By rewriting \eqref{eq:14} in terms of $\rho$, we get that, uniformly for $x \in A$,
	\begin{equation}\label{eq:16}
		\int_{\Gamma} \Big| \rho_{1-e^{-t}} \big( e^{-t/\alpha}x,y \big) - \vphi(y) \Big| M_{\Gamma}^2(y)\ud y \to 0 \quad \text{as} \quad t \to \infty.
	\end{equation}
This leads to	the following
spacial asymptotics for $\rho_1$.
	\begin{corollary}\label{lem:1}
		Let $\Gamma$ be a fat cone. If $\Gamma \ni x \to 0$ then $\int_{\Gamma} \big| \rho_1(x,y)-\vphi(y) \big|M_{\Gamma}^2(y)\ud y \to 0$.
	\end{corollary}
	
	\begin{proof}
		By the scalings \eqref{eq:scaling1} and \rf{eq:scaling2},
		\[
		\rho_{1-e^{-t}} \big( e^{-t/\alpha}x,y \big) = \big( 1-e^{-t} \big)^{-(d+2\beta)/\alpha} \rho_1 \Big( \big( e^t-1 \big)^{-1/\alpha}x, \big( 1-e^{-t} \big)^{-1/\alpha}y \Big),
		\]
		thus, in view of \eqref{eq:16},
		\begin{equation}\label{eq:17}
			\int_{\Gamma} \bigg| \rho_1 \Big( \big( e^t-1 \big)^{-1/\alpha}x, \big( 1-e^{-t} \big)^{-1/\alpha}y \Big) - \vphi(y) \bigg| M_{\Gamma}^2(y)\ud y \to 0 \quad \text{as} \quad t \to \infty.
		\end{equation}
	By the continuity of dilations in $L^1(\Rd)$,
		\begin{equation*}
			\int_{\Gamma} \bigg| \vphi \Big( \big( 1-e^{-t} \big)^{1/\alpha}y \Big) M_{\Gamma}^2(y) - \vphi(y)M_{\Gamma}^2(y) \bigg|\ud y \to 0 \quad \text{as} \quad t \to \infty.
		\end{equation*}
		Thus, by a change of variables in \eqref{eq:17} and the triangle inequality, we conclude that
		\[
		\int_{\Gamma} \bigg| \rho_1 \Big( \big(e^t-1\big)^{-1/\alpha}z,y \Big) - \vphi(y) \bigg|M_{\Gamma}^2(y)\ud y \to 0 \quad \text{as} \quad t \to \infty
		\]
		uniformly for all $z \in A$. To end the proof, we take $A = B_1$ and $x = \big(e^t-1\big)^{-1/\alpha}z$, where $t = \ln \big(1+|x|^{-\alpha}\big)$ and $z = x/|x| \in A$. .
	\end{proof}
	\begin{lemma}\label{lem:vphi_redef}
		After a modification on set of Lebesgue measure $0$, $\vphi$ is continuous on $\Gamma$ and 
		\[
		\vphi(y) \approx (1+|y|)^{-d-\alpha} \frac{\P_y(\tau_{\Gamma}>1)}{M_{\Gamma}(y)}, \quad y \in \Gamma.
		\]
	\end{lemma}
	\begin{proof}
By Corollary \ref{lem:1} and \eqref{eq:rho_bound}, 
		\[
		\vphi(y) \approx (1+|y|)^{-d-\alpha} \frac{\P_y(\tau_{\Gamma}>1)}{M_{\Gamma}(y)}
		\]
on $\Gamma$ less a set of Lebesgue measure zero. 		
		Theorem \ref{thm:stat_density} entails that $\vphi = L_1\vphi$ $a.e.$, so it suffices to verify that $L_1\vphi$ is continuous on $\Gamma$. To this end we note that $\ell_1(x,y)$ is continuous in $x,y \in \Gamma$. Next, by \eqref{eq:l_def} and \eqref{eq:rho_factor},
		\begin{align*}
			\ell_1(x,y) \approx \frac{\P_{e^{-1/\alpha}x}\big(\tau_{\Gamma}>1-e^{-1}\big)}{M_{\Gamma}\big(e^{-1/\alpha}x\big)}p_{1-e^{-1}}\big( e^{-1/\alpha}x,y \big) \frac{\P_y\big(\tau_{\Gamma}>1-e^{-1}\big)}{M_{\Gamma}(y)}, \quad x,y \in \Gamma.
		\end{align*}
		Let $R>1$. By \eqref{eq:hk} and \eqref{eq:5},
		\cite[Remark 3]{KBTGMR10} and \eqref{eq:surv_scaling}, and the homogeneity \eqref{eq:M_hom} of $M_{\Gamma}$, 
		\[
		\ell_1(x,y) \lesssim (1+|x|)^{-d-\alpha} \frac{\P_x(\tau_{\Gamma}>1)}{M_{\Gamma}(x)}, \quad x \in \Gamma, y \in \Gamma_R.
		\]
		By the dominated convergence theorem, $L_1\vphi$ is continuous on $\Gamma_R$. 
	\end{proof}
	In what follows, $\vphi$ denotes the continuous modification from Lemma \ref{lem:vphi_redef}.
	\begin{theorem}\label{thm:rho_conv}
		Let $\Gamma$ be a fat cone. For every $t>0$, uniformly in $y \in \Gamma$ we have
		\[
	\rho_t(0,y):=
		\lim_{\Gamma \ni x \to 0} \rho_t(x,y) = t^{-(d+2\beta)/\alpha}\vphi\big(t^{-1/\alpha}y\big).
		\]
	\end{theorem}
	
	\begin{proof}
		If $\beta=0$ then $\rho_t(x,y) = p_t(x,y)$ and the claim is simply the continuity property of the heat kernel $p_t$. Thus, we assume that $\beta>0$. 
		
		We only prove the claim for $t=1$; the extension to arbitrary $t$ is a consequence of the scaling \eqref{eq:scaling1}. 
		By \eqref{eq:scaling2} and the Chapman-Kolmogorov property, for $x,y \in \Gamma$,
		\begin{align*}
			\rho_1(x,y) &= 2^{(d+2\beta)/\alpha}\rho_2 \big( 2^{1/\alpha}x,2^{1/\alpha}y \big) \\ &= 2^{(d+2\beta)/\alpha} \int_{\Gamma} \rho_1 \big( 2^{1/\alpha}x,z \big) \rho_1 \big( z,2^{1/\alpha}y \big) M_{\Gamma}^2(z)\ud z.
		\end{align*}
	We will prove that, uniformly in $y \in \Gamma$,
		\begin{equation}\label{eq:6}
			\int_{\Gamma} \rho_1 \big( 2^{1/\alpha}x,z \big) \rho_1 \big( z,2^{1/\alpha}y \big) M_{\Gamma}^2(z)\ud z \to \int_{\Gamma} \vphi(z) \rho_1 \big( z,2^{1/\alpha}y \big) M_{\Gamma}^2(z)\ud z
		\end{equation}
	 as $\Gamma \ni x \to 0$. 
		To this end we first claim that there is $c \in (0,\infty)$ dependent only on $\alpha$ and $\Gamma$, such that for all $x \in \Gamma_1$ and $y \in \Gamma$,
		\begin{equation}\label{eq:23}
		\int_{\Gamma} \big| \rho_1 \big( 2^{1/\alpha}x,z \big)-\vphi(z) \big| \rho_1 \big( z,2^{1/\alpha}y \big) M_{\Gamma}^2(z)\ud z \leq c(1+|y|)^{-\beta}.
		\end{equation}
		Indeed, denote $\tilde{y} = 2^{1/\alpha}y$. By \eqref{eq:rho_factor}, Lemma \ref{lem:vphi_redef} and \eqref{M_P_comp}, there is $c>0$ such that for all $z,y \in \Gamma$ and $x \in \Gamma_1$,
		\[
		\big| \rho_1 \big( 2^{1/\alpha}x,z \big)-\vphi(z) \big| \rho_1 \big( z,\widetilde{y} \big) M_{\Gamma}^2(z) \lesssim (1+|z|)^{-d-\alpha}(1+|z-\widetilde{y}|)^{-d-\alpha} (1+|y|)^{\alpha-\beta}.
		\]
		We split the integral in \eqref{eq:23} into two integrals. For $z\in A:=B(\widetilde{y},|\widetilde{y}|/2)$ we use the fact that $|z| \approx |\widetilde{y}| \approx |y|$ and $1+|z-\widetilde{y}| \geq 1$, therefore
		\begin{align}\label{eq:24}
			\int_A \big| \rho_1 \big( 2^{1/\alpha}x,z \big)-\vphi(z) \big| \rho_1 \big( z,\widetilde{y} \big) M_{\Gamma}^2(z)\ud z &\lesssim |y|^{d} (1+|y|)^{-d-\beta} \leq (1+|y|)^{-\beta}. 
		\end{align}
	For $z\in \Gamma \setminus A$ we simply have $1+|z| \geq 1$, thus,
		\begin{align*}
			\int_{\Gamma \setminus A} \big| \rho_1 \big( 2^{1/\alpha}x,z \big)-\vphi(z) \big| \rho_1 \big( z,\widetilde{y} \big) M_{\Gamma}^2(z)\ud z &\lesssim (1+|y|)^{\alpha-\beta} \int_{\Gamma \setminus A} (1+|z-\widetilde{y}|)^{-d-\alpha} \ud z \\ &\lesssim (1+|y|)^{\alpha-\beta} \int_{|\widetilde{y}|/2}^{\infty} (1+r)^{-1-\alpha}\ud r \\ &\lesssim (1+|y|)^{-\beta}.
		\end{align*}
		Combining it with \eqref{eq:24}, we arrive at \eqref{eq:23}, as claimed.
		
		Let $\epsilon>0$. In view of \eqref{eq:23} and the fact that $\beta > 0$, there is $R\in (0,\infty)$ depending only on $\alpha$, $\beta$, $\Gamma$ and $\epsilon$ such that
		\begin{equation}\label{eq:7}
			\int_{\Gamma} \big| \rho_1 \big( 2^{1/\alpha}x,z \big)-\vphi(z) \big| \rho_1 \big( z,2^{1/\alpha}y \big) M_{\Gamma}^2(z)\ud z < \epsilon,
		\end{equation}
		provided that $y \in \Gamma \setminus \Gamma_R$. For $y \in \Gamma_R$, by \eqref{eq:rho_bound2} we get
		\begin{equation*}
			\int_{\Gamma} \big| \rho_1 \big( 2^{1/\alpha}x,z \big)-\vphi(z) \big| \rho_1 \big( z,2^{1/\alpha}y \big) M_{\Gamma}^2(z)\ud z \lesssim \int_{\Gamma} \big| \rho_1 \big( 2^{1/\alpha}x,z \big)-\vphi(z) \big| M_{\Gamma}^2(z)\ud z,
		\end{equation*}
		with the implied constant dependent only on $\alpha$, $\beta$, $\Gamma$ and $R$, but not otherwise dependent of $y$. Thus, by Corollary \ref{lem:1},
		\begin{equation}\label{eq:8}
			\int_{\Gamma} \big| \rho_1 \big( 2^{1/\alpha}x,z \big)-\vphi(z) \big| \rho_1 \big( z,2^{1/\alpha}y \big) M_{\Gamma}^2(z)\ud z < \epsilon
		\end{equation}
		for all $y \in \Gamma_R$ and $x \in \Gamma_1$ small enough. Putting \eqref{eq:8} together with \eqref{eq:7} we arrive at \eqref{eq:6}. Using the scaling property \eqref{eq:scaling2} and Theorem \ref{thm:stat_density},
		\begin{align*}
			\lim_{\Gamma \ni x \to 0} \rho_1(x,y) &= 2^{(d+2\beta)/\alpha}\int_{\Gamma} \vphi(z) \rho_1 \big( z,2^{1/\alpha}y \big) M_{\Gamma}^2(z)\ud z \\ &= \int_{\Gamma} \vphi(z) \rho_{1/2} \big( 2^{-1/\alpha}z,y \big) M_{\Gamma}^2(z)\ud z = L_{\ln 2} \vphi(y) = \vphi(y).
		\end{align*}
		The proof is complete. 
	\end{proof}

Note that by the symmetry of $\rho_t$, for $x \in \Gamma$,
\[
\rho_t(x,0) := \lim_{\Gamma \ni y \to 0}\rho_t(x,y) = \rho_t(0,x) = t^{-(d+2\beta)/\alpha} \vphi \big(t^{-1/\alpha}x\big).
\]

	Recall also that by \eqref{eq:rho_bound} and \eqref{eq:rho_density},
	\[
	\rho_1(x,y) \approx (1+|y|)^{-d-\alpha} \frac{\P_y(\tau_{\Gamma}>1)}{M_{\Gamma}(y)} \in L^1(M_{\Gamma}^2(y)\ud y).
	\]
	Thus, by Theorem \ref{thm:rho_conv} and the dominated convergence theorem,
	\begin{equation}\label{eq:vphi_density}
		\int_{\Gamma} \vphi(x)M_{\Gamma}^2(x)\,\ud x=1.
	\end{equation}

	Let us summarize the results of this section in one statement.
	\begin{theorem}\label{thm:rho_ex}
		Assume $\Gamma$ is a fat cone. Then the function $\rho$ has a continuous extension to $(0,\infty) \times (\Gamma \cup \{0\})\times (\Gamma \cup \{0\})$ and
		\begin{equation}\label{eq:27}
			\rho_t(0,y) := \lim_{\Gamma \ni x \to 0} \rho_t(x,y)\in (0,\infty), \quad t>0,\,y \in \Gamma,
		\end{equation}
		satisfies
		\begin{equation}\label{eq:28}
			\rho_t(0,y) = t^{-(d+2\beta)/\alpha} \rho_1(0,t^{-1/\alpha}y), \quad t>0, \,y \in \Gamma,
		\end{equation}
		and
		\begin{equation}\label{eq:29}
			\int_{\Gamma} \rho_t(0,y)\rho_s(y,z) M_{\Gamma}^2(y)\ud y = \rho_{t+s}(0,z), \quad s,t>0,\,z \in \Gamma.
		\end{equation}
	\end{theorem}
	\begin{proof}
		The existence of the limit \eqref{eq:27} and the scaling property \eqref{eq:28} are proved in Theorem \ref{thm:rho_conv}, see also Lemma \ref{lem:vphi_redef}. For the proof of \eqref{eq:29} we employ \eqref{eq:rho_ChK} to write
		\begin{align*}
			\rho_{t+s}(0,z) = \lim_{\Gamma \ni y \to 0} \rho_{t+s}(y,z) = \lim_{\Gamma \ni y \to 0} \int_{\Gamma} \rho_t(y,w) \rho_s(w,z) M_{\Gamma}^2(w)\ud w,
		\end{align*}
		and use \eqref{eq:rho_density}, \eqref{eq:scaling1}, \eqref{eq:rho_bound2}, and the dominated convergence theorem. 
		
		Thus, it remains to prove the continuity of $\rho$ on $(0,\infty) \times (\Gamma \cup \{0\}) \times (\Gamma \cup \{0\})$. By symmetry and the Chapman-Kolmogorov property \eqref{eq:rho_ChK} of $\rho_1$,
		\begin{equation}\label{eq:30}
		\rho_1(x,y) = \int_{\Gamma} \rho_{1/2}(x,z)\rho_{1/2}(y,z)M_{\Gamma}^2(z)\ud z, \quad x,y \in \Gamma.
		\end{equation}
		The continuity of $\rho_1$ on $\Gamma \times \Gamma$ together with Theorem \ref{thm:rho_conv}, for every $x_0,y_0 \in \Gamma \cup \{0\}$ we have $\rho_{1/2}(x,z) \to \rho_{1/2}(x_0,z)$ and $\rho_{1/2}(y,z) \to \rho_{1/2}(y_0,z)$ as $x \to x_0$ and $y \to y_0$. Moreover, \eqref{eq:rho_bound} entails that
		\[
		\rho_{1/2}(x,z)\rho_{1/2}(y,z)M_{\Gamma}^2(z) \leq c  (1+|z|)^{-2d-2\alpha},
		\]
		with the constant $c$ possibly dependent on $x_0$ and $y_0$. It follows by the dominated convergence theorem that
		\[
		\rho_1(x,y) \to \int_{\Gamma} \rho_{1/2}(x_0,z)\rho_{1/2}(y_0,z)M_{\Gamma}^2(z)\ud z,
		\]
		as $x \to x_0$ and $y \to y_0$ and in view of \eqref{eq:30}, it is an extension of $\rho_1$ to $(\Gamma \cup \{0\}) \times (\Gamma \cup \{0\})$, which will be denoted by the same symbol. It follows now from \eqref{eq:scaling1} that
		\[
		t^{-(d+2\beta)/\alpha} \rho_1 \big( t^{-1/\alpha}x,t^{-1/\alpha}y \big), \quad x,y \in \Gamma \cup \{0\},\, t>0,
		\]
		is a finite continuous extension of $\rho_t$ for every $t>0$. It remains to observe that the extension is unique and jointly continuous in $(t,x,y) \in \R_+ \times (\Gamma \cup \{0\}) \times (\Gamma \cup \{0\}$.
	\end{proof}

	\begin{corollary}
		We have
		$
	\rho_1(0,0) = \lim_{\Gamma \ni x,y \to 0}\rho_1(x,y) \in (0,\infty).
		$
	\end{corollary}
	\begin{proof}
By Theorem \ref{thm:rho_ex}, $\rho_1(0,0) = \lim_{\Gamma \ni y \to 0}\vphi(y)=: \vphi(0)$. Thus, the claim follows by Lemma \ref{lem:vphi_redef} and \cite[Lemma 4.2]{BB04}. 
	\end{proof}

	\begin{proof}[Proof of Theorem \ref{thm:main1}]
By \eqref{eq:27}, $\Psi_t(x) = \rho_t(0,x)M_{\Gamma}(x)$, $t>0$, $x \in \Gamma$. Thus, the existence of $\Psi_t$ is just a reformulation of \eqref{eq:27}. The scaling property \eqref{eq:18} follows immediately from \eqref{eq:28} and the homogeneity of the Martin kernel \eqref{eq:M_hom}, and \eqref{eq:psi_evol} is equivalent to \eqref{eq:29}. 
	\end{proof}
We conclude this part by rephrasing \eqref{eq:vphi_density} in terms of $\Psi_t$:
	\begin{equation}\label{e.nM}
		\int_{\Gamma}\Psi_t(x)M_{\Gamma}(x)\ud x=1, \quad t>0.
	\end{equation}
\subsection{Yaglom limit}\label{subsec:Y}
The above results quickly lead to calculation of the
Yaglom limit
for the stable process (conditioned to stay in a cone). 
Note that our proof is different from that in \cite{BoPaWa2018}. We also cover more general cones, e.g.,  $\R\setminus\{0\}$ and $\R^2 \setminus ([0,\infty) \times \{0\})$.
First, we obtain the following extension of \cite[Theorem 3.1]{BoPaWa2018}.
	\begin{corollary}\label{cor:survM_as}
		Let $\Gamma$ be a fat cone. For every $t>0$,
		\[
		\lim_{\Gamma \ni x \to 0} \frac{\P_x(\tau_{\Gamma}>t)}{M_{\Gamma}(x)} = C_1t^{-\beta/\alpha}\quad 	\mbox{where} \quad C_1 = \int_{\Gamma} \vphi(z)M_{\Gamma}(z)\ud z \in (0,\infty).
		\]
		\end{corollary}
	
	\begin{proof}
		It is enough to prove the claim for $t=1$; the general case follows by the scalings \eqref{eq:surv_scaling} and \eqref{eq:M_hom}. We have
		\[
		\frac{\P_x(\tau_{\Gamma}>1)}{M_{\Gamma}(x)} = \int_{\Gamma} \frac{p_1^{\Gamma}(x,y)}{M_{\Gamma}(x)}\ud y = \int_{\Gamma} \rho_1(x,y) M_{\Gamma}(y)\ud y, \quad x \in \Gamma.
		\]
We use \eqref{eq:rho_bound}, the dominated convergence theorem, and Theorem \ref{thm:rho_conv} to get the conclusion.
	\end{proof}

The first identity below is the 
Yaglom limit. 
	\begin{theorem}\label{thm:YL}
		Assume $\Gamma$ is a fat cone and $B$ is a bounded subset of $\Gamma$. Then, 
		\[\lim_{t \to \infty} \P_x \left( t^{-1/\alpha}X_t \in A\; \middle|\; \tau_{\Gamma}>t \right) = \mu(A), \quad A \subset \Gamma,
		\]
		uniformly in $x\in B$,
		where
		\[
		\mu(A) 
		:= \frac{1}{C_1}\int_A  \vphi(y)M_{\Gamma}(y)\ud y, \quad A \subset \Gamma.
		\]
		%
	\end{theorem}
	\begin{proof}
		By \eqref{eq:surv_def} and the scaling property \eqref{eq:surv_scaling},
		\begin{align*}
			\P_x \left( t^{-1/\alpha}X_t \in A \;\middle|\; \tau_{\Gamma}>t \right) &= \frac{\P_x \left( \tau_{\Gamma}>t,t^{-1/\alpha}X_t \in A \right)}{\P_x(\tau_{\Gamma}>t)} \\ &= \frac{\P_{t^{-1/\alpha}x} \left( \tau_{\Gamma}>1,X_1 \in A \right)}{\P_{t^{-1/\alpha}x}(\tau_{\Gamma}>1)} \\ &= \int_A \frac{p_1^{\Gamma}\big(t^{-1/\alpha}x,y\big)}{M_{\Gamma}\big(t^{-1/\alpha}x\big)}\ud y \cdot \frac{M_{\Gamma}(t^{-1/\alpha}x)}{\P_{t^{-1/\alpha}x}(\tau_{\Gamma}>1)}.
		\end{align*}
		The claim follows by Corollary \ref{cor:survM_as}, \eqref{eq:rho_bound}, and the dominated convergence theorem.
	\end{proof}
	
	\begin{theorem}\label{thm:UYL}
		If $\Gamma$ is a fat cone and $\gamma$ is a probability measure on $\Gamma$ with $\int_{\Gamma}(1+|y|)^{\alpha}\,\gamma({\rm d}y)<\infty$, then
		\[
		\lim_{t \to \infty} \P_\gamma \left( t^{-1/\alpha}X_t \in A\; \middle|\; \tau_{\Gamma}>t \right) = \mu(A), \quad A \subset \Gamma.
		\]
	\end{theorem}
	\begin{proof}
		Let $t\ge 1$.	In view of \cite[Lemma 8.7]{Kallenberg02}, we may write
		\begin{align*}
			\P_\gamma \left( t^{-1/\alpha}X_t \in A\; \middle|\; \tau_{\Gamma}>t \right) &= \frac{\P_\gamma \left( t^{-1/\alpha}X_t \in A, \tau_{\Gamma}>t \right)}{\P_\gamma \left(  \tau_{\Gamma}>t \right)} \\ 
			&= \int_{\Gamma} \P_x \left( t^{-1/\alpha}X_t \in A\; \middle|\; \tau_{\Gamma}>t \right)
			\frac{\P_x \left(  \tau_{\Gamma}>t \right)}{\P_\gamma \left(  \tau_{\Gamma}>t \right)}\, \gamma(\od x)	.	
			\nonumber
		\end{align*}
		We first prove that for all $x \in \Gamma$,
		\begin{equation}\label{eq:11}
			\frac{\P_\gamma \left(  \tau_{\Gamma}>t \right)}{\P_x \left(  \tau_{\Gamma}>t \right)}:= \int_{\Gamma} \frac{\P_y \left(  \tau_{\Gamma}>t \right)}{\P_x \left(  \tau_{\Gamma}>t \right)}\, \gamma(\ud y) \to  \int_{\Gamma} \frac{M_{\Gamma}(y)}{M_{\Gamma}(x)}\, \gamma(\od y),
		\end{equation}
		as $t \to \infty$.
		Indeed, fix $x \in \Gamma$. First we note that by local boundedness of $M_{\Gamma}$ and \eqref{eq:M_hom},
		\[
		\int_{\Gamma} M_{\Gamma}(y) \,\gamma({\rm d}y) \leq c\int_{\Gamma}(1+|y|)^{\beta}\, \gamma({\rm d}y) <\infty,
		\]
		so the right-hand side of \eqref{eq:11} is finite. Next, by Corollary \ref{cor:survM_as}, \eqref{eq:surv_scaling}, and \eqref{eq:M_hom}, 
		\begin{align*}
			\lim_{t \to \infty }\frac{\P_y \left(  \tau_{\Gamma}>t \right)}{\P_x \left(  \tau_{\Gamma}>t \right)} =
			\lim_{t \to \infty }\frac{\P_{t^{-1/\alpha}y} \left(  \tau_{\Gamma}>1 \right)M_{\Gamma}(t^{-1/\alpha}x)}{\P_{t^{-1/\alpha}x} \left(  \tau_{\Gamma}>1 \right)M_{\Gamma}(t^{-1/\alpha}y)} \frac{M_{\Gamma}(y)}{M_{\Gamma}(x)} =		
			\frac{M_{\Gamma}(y)}{M_{\Gamma}(x)}, \quad x,y \in \Gamma. 
		\end{align*}
		Moreover, since $x$ is fixed, we may assume that $t \geq 1\vee |x|^{\alpha}$. Thus, by \cite[Lemma 4.2]{BB04}, Lemma \ref{lem:sur_M_comp}, the local boundedness of $M_{\Gamma}$ and \eqref{eq:M_hom},
		\begin{equation*}
			\frac{\P_y \left(  \tau_{\Gamma}>t \right)}{\P_x \left(  \tau_{\Gamma}>t \right)} \leq c \frac{\big(t^{-\beta/\alpha} + t^{-1}|y|^{\alpha-\beta}\big)M_{\Gamma}(y)}{t^{-\beta/\alpha}M_{\Gamma}(x)} \leq c \frac{\big(1+|y|^{\alpha-\beta}\big)M_{\Gamma}(y)}{M_{\Gamma}(x)} \leq c \frac{(1+|y|)^{\alpha}}{M_{\Gamma}(x)}.
		\end{equation*}
		Thus, the dominated convergence theorem yields \eqref{eq:11}, as desired.
	
	Next, we consider a family $\calF_1$ of functions $f_t$ of the form
	\[
	f_t(x) =  \frac{\P_x(\tau_{\Gamma}>t)}{\P_{\gamma}(\tau_{\Gamma}>t)}, \quad x \in \Gamma, \, t \geq 1.
	\]
	Denote
	\[
	f(x) = \frac{M_{\Gamma}(x)}{\int_{\Gamma}M_{\Gamma}(y)\,\gamma({\rm d}y)}, \quad x \in \Gamma.
	\]
	By virtue of \eqref{eq:11}, $f_t \to f$ everywhere in $\Gamma$ as $t \to \infty$. 
	Thus, $f_t \to f$ in measure $\gamma$ as $t \to \infty$, see \cite[Definition 22.2]{MR3644418}. Moreover, we have
	\[
	\int_{\Gamma} f(x)\,\gamma({\rm d}x) = 1 = \lim_{t \to \infty} 1 = \lim_{t \to \infty} \int_{\Gamma}f_t(x)\,\gamma({\rm d}x).
	\]
	Therefore, by Vitali's theorem \cite[Theorem 22.7]{MR3644418}, the family $\calF_1$ is uniformly integrable. If we now consider the family $\calF_2$ of functions $\tilde{f}_t$ of the form
	\[
	\tilde{f}_t(x) = \P_x \left( t^{-1/\alpha}X_t \in A\; \middle|\; \tau_{\Gamma}>t \right) f_t(x), \quad x \in \Gamma, \, t \geq 1,
	\]
	then a trivial bound $\P_x \left( t^{-1/\alpha}X_t \in A\; \middle|\; \tau_{\Gamma}>t \right) \leq 1$ shows that $\calF_2$ is uniformly integrable as well (see, e.g., \cite[Theorem 22.9]{MR3644418}). By Theorem \ref{thm:YL}, \eqref{eq:11} and \cite[Theorem 22.7]{MR3644418},
	\[
	\lim_{t \to \infty}\int_{\Gamma} \P_x \left( t^{-1/\alpha}X_t \in A\; \middle|\; \tau_{\Gamma}>t \right)
	\frac{\P_x \left(  \tau_{\Gamma}>t \right)}{\P_\gamma \left(  \tau_{\Gamma}>t \right)}\, \gamma(\od x) = \int_{\Gamma} \mu(A) \frac{M_{\Gamma}(x)}{\int_{\Gamma}M_{\Gamma}(y)\,\gamma({\rm d}y)} \,\gamma({\rm d}x) = \mu(A).
	\]
	The proof is complete.
	\end{proof}
	\begin{example}\label{rem:beta0}
Note that $\beta=0$ if and only if $\Gamma^c$ is a polar set and then $M_{\Gamma}(x)=1$ for all $x \in \Gamma$, see \cite[Theorem 3.2]{BB04}. Consequently, we have $p_t^{\Gamma}(x,y) = p_t(x,y)$ and $\P_x(\tau_{\Gamma}>t)=1$ for all $x,y \in \Gamma$ and all $t>0$. It follows that $\rho_t(x,y) = p_t(x,y)$ and a direct calculation using the Chapman-Kolmogorov property entails that $\vphi(y) = p_1(0,y)$ is the stationary density for the (classical) $\alpha$-stable Ornstein-Uhlenbeck semigroup, see \eqref{eq:l_def} and Theorem \ref{thm:stat_density}. The statement of Theorem \ref{thm:rho_conv} thus reduces to continuity of the heat kernel of the isotropic $\alpha$-stable L\'{e}vy process. Theorems \ref{thm:YL} and \ref{thm:UYL} trivialize in a similar way. Incidentally, in this case the moment condition on $\gamma$ in Theorem \ref{thm:UYL} is superfluous. Further examples are given in Section~\ref{s:Ab}.
	\end{example}

	\section{Asymptotic behavior for the killed semigroup}\label{s:Ab}
	This section is devoted to examples and applications in Functional Analysis and Partial Differential Equations. Note that in Lemmas \ref{lem:contraction} and \ref{lem:hyperc} we do not assume that $\Gamma$ is fat.
	\begin{lemma}\label{lem:contraction}
		$\{P_t^\Gamma \}_{t>0}$
		is a strongly continuous contraction semigroup on $L^1(M_\Gamma)$ and 
		\begin{equation}\label{e.cl}
			\int_{\Gamma} P_t^{\Gamma} f(x) M_{\Gamma}(x)\ud x=\int_{\Gamma} f(x)M_\Gamma (x)\ud x,\quad t>0,\quad f\in L^1(M_\Gamma).
		\end{equation}
	\end{lemma}

	\begin{proof}
	 Let $f \geq 0$. By the Fubini-Tonelli theorem, the symmetry of $p_t^{\Gamma}$ and Theorem \ref{thm:M_invariance},
\begin{equation}\label{e.ipf}
	\int_{\Gamma} P_t^{\Gamma} f(x) M_{\Gamma}(x)\ud x = \int_{\Gamma}\int_{\Gamma} p_t^{\Gamma}(x,y) f(y) M_{\Gamma}(y)\ud y \ud x = \int_{\Gamma} f(y)M_{\Gamma}(y)\ud y.
	\end{equation}
Since $|P_t^\Gamma f|\le P_t^\Gamma|f|$, the contractivity follows. Furthermore, for arbitrary $f \in L^1(M_{\Gamma})$ we write $f= f_+ - f_-$ and use \eqref{e.ipf} to prove \eqref{e.cl}. The semigroup property follows from \eqref{eq:9}.

	To prove the strong continuity, we fix  $f \in L^1(M_{\Gamma})$ and let $G:=fM_{\Gamma} \in L^1(\Gamma)$. There is a sequence $g_n \in C_c^{\infty}(\Gamma)$ such that $\|g_n-G\|_{L^1(\Gamma)} \to 0$ as $n \to \infty$. For $f_n := g_n/M_{\Gamma}$ we get $f_n \in C_c^{\infty}(\Gamma)$ and $\|f_n-f \|_{L^1(M_{\Gamma})}=\|g_n-G\|_{L^1(\Gamma)} \to 0$. By the first part of the proof, 
	\begin{align*}
	\|P_t^{\Gamma}f-f\|_{L^1(M_{\Gamma})} &\leq \|P_t^{\Gamma}f-P_t^{\Gamma}f_n\|_{L^1(M_{\Gamma})}+\|P_t^{\Gamma}f_n-f_n\|_{L^1(M_{\Gamma})}+\|f_n-f\|_{L^1(M_{\Gamma})} \\ &\leq 2\|f_n-f\|_{L^1(M_{\Gamma})} + \|P_t^{\Gamma}f_n-f_n\|_{L^1(M_{\Gamma})}.
	\end{align*}
It remains to prove that $\|P_t^{\Gamma}f-f\|_{L^1(M_{\Gamma})} \to 0$ as $t \to 0^+$ for every $f \in C_c^{\infty}(\Gamma)$. To this end we let $\epsilon > 0$ and choose $R >0$ such that $\supp f \in B_R$ and
	$\int_{\Gamma \setminus \Gamma_R} P_t^{\Gamma}|f|(x)M_{\Gamma}(x)\ud x < \epsilon$.
	Then,
	\begin{equation}\label{e.on}
		\| P_t^{\Gamma}f-f\|_{L^1(M_{\Gamma})} < \int_{\Gamma_R} \big| P_t^{\Gamma}f(x)-f(x) \big| M_{\Gamma}(x)\ud x + \epsilon.
	\end{equation}
Considering the integrand in \eqref{e.on}, for all $x\in \Gamma_R$ we have
	\begin{equation}\label{eq:25}
	\big| P_t^{\Gamma}f(x)-f(x) \big| \leq \int_{\Gamma} p_t^{\Gamma}(x,y)|f(y)-f(x)|\ud y + |f(x)|\P_x(\tau_{\Gamma} \leq t).
	\end{equation}
	Since $P_t f \to f$ uniformly as $t\to 0^+$, for $t>0$ small enough we get 
	\begin{align*}
		\int_{\Gamma} p_t^{\Gamma}(x,y)|f(y)-f(x)|\ud y \leq \int_{\Rd} p_t(x,y)|f(y)-f(x)|\ud y < \epsilon.
	\end{align*}
On the other hand, 
	\[
	|f(x)|\P_x(\tau_{\Gamma} \leq t) \leq 
\| f\|_{\infty} \sup_{x \in K} \P_x(\tau_{\Gamma} \leq t),
	\]
where $K:=\supp f$. We have $r:= \dist(K,\Gamma^c)>0$, so
	\[
	\P_x(\tau_{\Gamma} \leq t) \leq \P_x(\tau_{B(x,r)} \leq t) = \P_0(\tau_{B_r} \leq t) \leq ctr^{-\alpha}<\epsilon,
	\]
	for $t$ small enough, see, e.g., \cite{Pruitt81}. 	By \eqref{e.on} and \eqref{eq:25}  we get, as required, 
	\[
		\| P_t^{\Gamma}f-f\|_{L^1(M_{\Gamma})} < \epsilon+(\epsilon+\|f\|_\infty\epsilon) |\Gamma_R|\sup_{\Gamma_R} M_{\Gamma}.
	\]
	\end{proof}
Recall that
\begin{equation*}
	\|f\|_{q,M_\Gamma }:=\|f/M_\Gamma\|_{L^q(M_\Gamma^2)}
	=\bigg( \int_{\Gamma} |f(x)|^q M_\Gamma^{2-q}(x)\ud x \bigg) ^{\frac{1}{q}}
	=\|f\|_{L^q(M_\Gamma^{2-q})},
\end{equation*}      
if $1 \leq q < \infty$, and
\begin{equation*}
	\|f\|_{\infty, M_\Gamma}:={\rm ess} \sup_{x\in \Gamma} |f(x)|/M_\Gamma(x).
\end{equation*}
	The following characterization of \textit{hypercontractivity} of $P^\Gamma_t$ is crucial for the proof of \eqref{e.sc}.
	\begin{lemma}\label{lem:hyperc}
		Let $q \in [1,\infty)$. We have 
		\begin{equation}\label{eq:w_norm}
			\| P_t^\Gamma f\|_{q,M_\Gamma} \le  Ct^{-\frac{d+2\beta}{\alpha}\frac{q-1}{q}} \|f\|_{1,M_\Gamma }
		\end{equation}
for all $t>0$ and all non-negative functions $f$ on $\Rd$  if and only if
		\begin{equation}\label{eq:20}
		\sup_{y\in \Gamma} \int_{\Gamma} \rho_1(x,y)^q M_{\Gamma}^2(x)\ud x <\infty.
	\end{equation}	
		
	\end{lemma}
	
\begin{proof}
	Assume \eqref{eq:20}. Let $f\ge 0$. With the notation $F:=f/M_{\Gamma}$ we get
	\[
	\| P_1^{\Gamma}f \|_{q,M_{\Gamma}} = \Bigg(\int_{\Gamma} \bigg( \int_{\Gamma} \rho_1(x,y) F(y)M_{\Gamma}^2(y)\ud y\bigg)^q M_{\Gamma}^2(x)\ud x\Bigg)^{1/q}.
	\]
Let $c$ be the supremum in \eqref{eq:20}.
By Minkowski integral inequality, 
	\begin{align*}
	\Bigg(\int_{\Gamma} \bigg( \int_{\Gamma} \rho_1(x,y) F(y)M_{\Gamma}^2(y)\ud y\bigg)^q M_{\Gamma}^2(x)\ud x \Bigg)^{1/q} &\leq \int_{\Gamma} \bigg(\int_{\Gamma} \rho_1(x,y)^qM_{\Gamma}^2(x) \ud x\bigg)^{1/q} F(y)M_{\Gamma}^2(y) \ud y \\ &\le c \int_{\Gamma} F(y)M_{\Gamma}^2(y) \ud y = c\|f\|_{1,M_{\Gamma}}.
	\end{align*}
For $t>0$, by scaling we get \eqref{eq:w_norm} as follows:
\begin{align*}
	\| P_t^{\Gamma} f \|_{q,M_{\Gamma}} &= t^{\frac{d+\beta(2-q)}{\alpha q}} \| P_1^{\Gamma} f(t^{1/\alpha}\,\cdot\,) \|_{q,M_{\Gamma}} \le c t^{\frac{d+\beta(2-q)}{\alpha q}}  \| f(t^{1/\alpha}\,\cdot\,) \|_{1,M_{\Gamma}} \\ &= ct^{\frac{d+\beta(2-q)}{\alpha q}}  t^{-\frac{d+\beta}{\alpha}}\| f \|_{1,M_{\Gamma}} = ct^{-\frac{d+2\beta}{\alpha}\frac{q-1}{q}} \| f\|_{1,M_{\Gamma}}.
\end{align*}
Conversely, assume \eqref{eq:w_norm}. Let $y \in \Gamma$. Let $g_n\ge 0$, $n\in \N$, be functions in $C^\infty_c(\Gamma)$ approximating $\delta_y$, the Dirac measure at $y$, as follows:
\[
\int_{\Gamma} g_n(x)\ud x = 1, \qquad \text{and} \qquad \lim_{n \to \infty} \int_{\Gamma} h(x)g_n(x)\ud x = h(y), 
\]
for every function $h$ continuous near $y$. For $f_n := g_n/M_{\Gamma}$, $\| f_n\|_{1,M_{\Gamma}} = \| g_n \|_1 = 1$ and
\[
P_1^{\Gamma} f_n(x) = \int_{\Gamma} p_1^{\Gamma}(x,z)\frac{g_n(z)}{M_{\Gamma}(z)} \ud z \to \frac{p_1^{\Gamma}(x,y)}{M_{\Gamma}(y)},
\]
as $n \to \infty$. By \eqref{eq:w_norm} and Fatou's lemma,
\begin{align*}
	C^q &\geq \liminf_{n \to \infty} \|P_1^{\Gamma} f_n\|_{q,M_{\Gamma}}^q \\ &=\liminf_{n \to \infty} \int_{\Gamma} \bigg| \int_{\Gamma} \frac{p_1^{\Gamma}(x,z)}{M_{\Gamma}(z)}g_n(z)\ud z \bigg|^q M_{\Gamma}^{2-q}(x)\ud x \\ &\geq \int_{\Gamma} \rho_1(x,y)^qM_{\Gamma}^2(x)\ud x.
\end{align*}
Since $y \in \Gamma$ was arbitrary, we obtain \eqref{eq:20}.
\end{proof}
\begin{remark}
Of course, \eqref{eq:w_norm} extends to arbitrary $f\in L^1(M_\Gamma)$.
	\end{remark}
	
\begin{example}\label{ex:beta0} As in Example \ref{rem:beta0}, we assume that $\beta=0$. In fact, to simplify notation, let $\Gamma=\Rd$. Then \eqref{eq:20} is trivially satisfied for every $q\in [1,\infty)$, because $\rho_1(x,y)=p_1(x,y)$ is bounded. Therefore, by \eqref{eq:w_norm}, for each $f\in L^1$,
	\begin{equation*}
			\| P_t f\|_{q} \le  Ct^{-\frac{d}{\alpha}\frac{q-1}{q}} \|f\|_{1}.
		\end{equation*}
This agrees with \cite{Vasquez18}, see also \cite{BJKP22}.		
	\end{example}
Here is a refinement of Lemma~\ref{lem:hyperc}.
	\begin{lemma}\label{lem_lim_0}
		 Let $q \in [1,\infty)$, assume \eqref{eq:20} and suppose $\Gamma$ is fat. If $f\in L^1(M_\Gamma)$, $\int_{\Gamma } f(x) M_\Gamma (x)\ud x  =0$ then 
		\begin{equation}\label{lim_0_cond}
			\lim_{t \rightarrow \infty }t^{\frac{d+2\beta}{\alpha}\frac{q-1}{q}}\|  P_t^{\Gamma} f\|_{q,M_\Gamma }=0.
		\end{equation}
		If, additionally, $f$ has compact support, then \eqref{lim_0_cond} is true for $q=\infty$, too.
	\end{lemma}
	
	\begin{proof}
		Let $\omega>0$. First, we prove prove \rf{lim_0_cond} for a compactly supported function $f \in L^1(M_\Gamma)$ satisfying 
		\begin{equation}\label{0 con 1}
			\int_{\Gamma} f (x) M_\Gamma (x)  \ud x =0.
		\end{equation}
%
		\noindent
		{\it Step 1. Case $q=\infty $.}\\
		For $t>0$ we let
		\begin{align*}
			I(t)&:= t^{\frac{d+2\beta}{\alpha}}\|  P_t^\Gamma f \|_{\infty ,M_\Gamma }
			=t^{\frac{d+2\beta}{\alpha}}\sup_{x\in \Gamma} \left| \int_{\Gamma} \rho _t(x,y)M_\Gamma (y)f(y) \ud y \right|.
		\end{align*}
		By \rf{0 con 1},
		\begin{align*}
			I(t) = t^{\frac{d+2\beta}{\alpha}}\sup_{x\in \Gamma }\left| \int_{\Gamma} \left( \rho_t(x,y)-\rho_t(x,0)\right)M_\Gamma (y)f (y)  \ud y \right|.
		\end{align*}
		Since $f$ has compact support,  for sufficiently large $t>0$ we have
		\begin{align*}
			I(t) &=t^{\frac{d+2\beta}{\alpha}}\sup_{x\in \Gamma }\left| \int_{|y|\leq t^{\frac{1}{\alpha}}\omega} \left( \rho_t(x,y)-\rho_t(x,0)\right)M_\Gamma (y)f (y)  \ud y \right|\\
			& \le   t^{\frac{d+2\beta}{\alpha}}\sup_{\substack {x\in \Gamma\\ |y|\leq t^{\frac{1}{\alpha}}\omega }}\big|\rho_t(x,y) -\rho_ t(x,0)\big| \int_{|y|\leq t^{\frac{1}{\alpha}}\omega}M_\Gamma (y)|f (y)|  \ud y \\ &=    \sup_{\substack {x\in \Gamma\\ |y|\leq \omega }}\left| \rho_1\left(x,y \right) - \rho_1\left( x, 0\right) \right| \int_{\Gamma}M_\Gamma (y)| f(y)|  \ud y,
		\end{align*}
		where in the last line we used scaling \eqref{eq:scaling1} of $\rho$.
		By Theorem \ref{thm:rho_conv}, we can make it arbitrary small by choosing small $\omega$, and \eqref{lim_0_cond} follows in this case. 
\\		
		{\it Step 2. Case $q=1$. }\\
		For $t>0$ we let
		\begin{equation*}
			J(t):=\| P_t^\Gamma f \|_{1,M_\Gamma }= \int_{\Gamma}\left| \int_{\Gamma} p^\Gamma_t (x,y)M_\Gamma (x)f(y) \ud y \right| \ud x  =\int_{\Gamma }\Big| \int_{\Gamma }\rho_t (x,y)M_\Gamma^2(x)f (y)M_\Gamma (y)  \ud y  \Big|   \ud x .
		\end{equation*}
		Applying \rf{0 con 1}, we get
		\begin{equation*}
			J(t) \le  \int_{\Gamma } \int_{\Gamma } \big| \rho_t(x,y)-\rho_t(x,0)\big|M_\Gamma ^2(x) |f (y)|M_\Gamma (y)  \ud y   \ud x .
		\end{equation*}
		Since $f$ has compact support,
		\begin{align*}
			J(t)& \le  \int_{\Gamma } \int_{|y|\leq t^{\frac{1}{\alpha}}\omega } \big| \rho_t (x,y)-\rho_t(x,0)\big|M_\Gamma ^2(x) |f (y)| M_\Gamma (y)  \ud y   \ud x \\ & \le  \sup_{|y|\leq t^{\frac{1}{\alpha}}\omega}\| \rho_t(\cdot , y)-\rho_t(\cdot ,0)\|_{L^1(M_\Gamma^2)} \int_{\Gamma } M_\Gamma (y)|f (y)|  \ud y,
		\end{align*}
		for sufficiently large $t$. 
In view of \eqref{eq:M_hom} and \eqref{eq:scaling1}, by changing variables $t^{-1/\alpha}x \to x$ and $t^{-1/\alpha}y \to y$ we obtain
		\begin{align*}
			&\sup_{|y|\leq t^{\frac{1}{\alpha}}\omega}\| \rho_t(\cdot , y)-\rho_t(\cdot ,0)\|_{L^1(M_\Gamma^2)} \\
			&= t^{-\frac{d+2\beta}{\alpha}} \sup_{|y|\leq t^{\frac{1}{\alpha}}\omega }\int \left | \rho_1 \left(t^{-{1}/{\alpha}}x,t^{-{1}/{\alpha}}y \right) -\rho_1\left( t^{-{1}/{\alpha}}x,0\right)\right|M_\Gamma ^2(x) \ud x \\
			&=\sup_{|y|\leq \omega} \| \rho_1(\cdot ,y) -\rho_1(\cdot, 0)\|_{L^1(M_\Gamma^2)}.
		\end{align*}
By Corollary~\ref{lem:1}, we can make it arbitrary small by choosing small $\omega$, so \rf{lim_0_cond} is true.\\
{\it Step 3. Case $q\in (1,\infty)$. }\\
		By H\"older inequality we get that, as $t \to \infty$,
		\begin{align*}
			t^{\frac{d+2\beta}{\alpha}\frac{q-1}{q}}\|P_t^\Gamma f \|_{q, M_\Gamma }&=t^{\frac{d+2\beta}{\alpha}\frac{q-1}{q}}\bigg(\int_{\Gamma }\big|  P_t^\Gamma f (x)/M_\Gamma (x)\big|^{q-1}\big| P_t^\Gamma f (x)M_\Gamma (x)\big|  \ud x \bigg)^{\frac{1}{q}} \\
			& \le  \left(t^{\frac{d+2\beta}{\alpha}}\|P_t^\Gamma f \|_{\infty ,M_\Gamma}\right)^{\frac{q-1}{q}} \| P_t^\Gamma f \|_{1, M_\Gamma}^{\frac{1}{q}} \to 0,
		\end{align*}
		since both factors converge to zero as $t\rightarrow \infty $ by {\it Step 1.} and {\it Step 2.} 

Finally, consider arbitrary $f\in L^1(M_\Gamma)$ with $\int_{\Gamma } f(x) M_\Gamma (x)\ud x  =0$.	
Let $R>0$ and $f_R (x)= (f (x)-c_R)\indyk_{|x| \le  R}$, where $c_R=\int_{|x|\le R} f(x)M_\Gamma (x)\ud x/ \int_{|x|\le R} M_\Gamma (x)\ud x$. Of course, 
\begin{equation}\label{e.kpsi}
\int_{\Gamma } M_\Gamma (x) f_R (x)  \ud x  =0,
\end{equation} 
and $f_R$ is compactly supported. Furthermore, due to our assumptions,
		\begin{align*}
			\|f- f_R \|_{L^1(M_\Gamma )}& =   |c_R| \int_{|x| \le  R}M_\Gamma (x)  \ud x+\int_{|x|> R}M_\Gamma (x)|f (x)|  \ud x  \\
			& = \Big| \int_{|x| \le  R} M_\Gamma (x)f (x)  \ud x  \Big| + \int_{|x|> R}M_\Gamma (x)|f (x)|  \ud x \to 0 
		\end{align*}
		as $R\rightarrow \infty$.
Let $\ve >0$ and choose $R>0$ so large that 
		\begin{equation*}
			\| f -f_R \|_{1,M_\Gamma}<\ve.
		\end{equation*}
		
		For $q=1$, by using the triangle inequality and Lemma \ref{lem:contraction}, we get
		\begin{align*}
			\|P_t^\Gamma f\|_{1,M_\Gamma }& \le  \|P_t^\Gamma f_R \|_{1,M_\Gamma  }+ \|P_t^\Gamma(f-f_R) \|_{1,M_\Gamma}\\
			& \le  \|P_t^\Gamma f_R \|_{1,M_\Gamma }+ 
			\|f -f_R \|_{1, M_\Gamma } ,
		\end{align*}
		and {\it Step 2.} yields
		\begin{equation*}
			\limsup_{t\rightarrow \infty } \|P_t^\Gamma f\|_{1,M_\Gamma } \le  \ve,
		\end{equation*}
which proves \rf{lim_0_cond} in this case.
		
		If $1<q<\infty$, then using the triangle inequality and Lemma \ref{lem:hyperc}, we obtain 
		\begin{align*}
			t^{\frac{d+2\beta}{\alpha}\frac{q-1}{q}}\|P_t^\Gamma f\|_{q,M_\Gamma }& \le  t^{\frac{d+2\beta}{\alpha}\frac{q-1}{q}}\|P_t^\Gamma f_R \|_{q,M_\Gamma}+ t^{\frac{d+2\beta}{\alpha}\frac{q-1}{q}}\|P_t^\Gamma(f-f_R) \|_{q,M_\Gamma}\\
			& \le  t^{\frac{d+2\beta}{\alpha}\frac{q-1}{q}}\|P_t^\Gamma f_R \|_{q,M_\Gamma }+ 
			C\|f -f_R \|_{1, M_\Gamma }. 
		\end{align*}
By \eqref{e.kpsi} and {\it Step 3.},
		\begin{equation*}
			\limsup_{t\rightarrow \infty } t^{\frac{d+2\beta}{\alpha}\frac{q-1}{q}}\|P_t^\Gamma f\|_{q,M_\Gamma } \le  
			2C\ve.
		\end{equation*}
This completes the proof of \rf{lim_0_cond} for $q\in (1,\infty)$. 
	\end{proof}

	\begin{theorem}\label{thm_lim_norm}
		Let $q\in[1,\infty)$, assume \eqref{eq:20} and suppose $\Gamma$ is fat. Then for $f\in L^1(M_\Gamma)$ and $A=\int_{\Gamma}f (x)  M_\Gamma(x) \ud x$, 
		\begin{equation*}
			\lim_{t\rightarrow \infty }t^{\frac{d+2\beta}{\alpha}\frac{q-1}{q}}\| P_t^\Gamma f-A\Psi_t\|_{q, M_\Gamma}=0.
		\end{equation*}
	\end{theorem}
	\begin{remark}\label{rem:A_constant}
		In view of \eqref{e.nM} and Lemma \ref{lem:contraction}, the constant $A$ in Theorem \ref{thm_lim_norm} satisfies
		\[
		\int_{\Gamma} \big(P_t^{\Gamma} f(x) - A \Psi_s(x)\big) M_{\Gamma}(x)\ud x = 0, \quad s, t >0.
		\]
	\end{remark}
	\begin{proof}[Proof of Theorem \ref{thm_lim_norm}]
		By \eqref{eq:9}, \eqref{eq:psi_evol}, Remark \ref{rem:A_constant} and Lemma \ref{lem_lim_0},
		\begin{align*}
			\lim_{t\rightarrow \infty }t^{\frac{d+2\beta}{\alpha}\frac{q-1}{q}}\| P_t^\Gamma f-A\Psi_t\|_{q, M_\Gamma} &= \lim_{t\rightarrow \infty }t^{\frac{d+2\beta}{\alpha}\frac{q-1}{q}}\| P_{t+1}^\Gamma f-A\Psi_{t+1}\|_{q, M_\Gamma} \\ &= \lim_{t\rightarrow \infty }t^{\frac{d+2\beta}{\alpha}\frac{q-1}{q}}\| P_t^{\Gamma} \big( P_{1}^\Gamma f-A\Psi_{1} \big)\|_{q, M_\Gamma} =0.
		\end{align*}
	\end{proof}

\subsection{Applications}
We conclude the article by providing several applications and examples which apply to our results. In particular, we draw the reader's attention to Lemma \ref{lem:rc_cones}, which provides sharp distinction between cones bigger and smaller than the half-space $\mathbb{R}_+^d:= \{x=(x_1,\ldots,x_d) \in \Rd \colon x_d > 0\}$. The same behavior is displayed by the bigger class of smooth cones, as we assert in Corollary \ref{cor:s_cones}. First, we note a simple observation.
	\begin{example}
		Let $q=1$. By \eqref{eq:rho_density}, the condition \eqref{eq:20} holds for every fat cone $\Gamma$.
	\end{example}
	\begin{lemma}\label{lem:rc_cones}
		Let $q \in (1,\infty)$ and suppose $\Gamma$ is a right-circular cone. Then \eqref{eq:20} holds if $\beta \geq \alpha/2$. Conversely, if $d \geq 2$ and $\beta<\alpha/2$, then \eqref{eq:20} does not hold.
	\end{lemma}
	\begin{proof}
		Recall that by \cite[Theorem 2 and Eq. (3)]{KBTGMR10},
		\begin{equation}\label{eq:21}
			p_1^{\Gamma}(x,y) \approx p_1(x,y) \frac{\big(1 \wedge \delta_{\Gamma}(x)\big)^{\alpha/2} \big(1 \wedge \delta_{\Gamma}(y)\big)^{\alpha/2}}{\big(1 \wedge |x|)^{\alpha/2-\beta} \big(1 \wedge |y|)^{\alpha/2-\beta}}, \quad x,y \in \Gamma.
		\end{equation}
		Moreover, \cite[Lemma 3.3]{Michalik06} entails that
		\[
		M_{\Gamma}(x) \approx \delta_{\Gamma}(x)^{\alpha/2}|x|^{\beta-\alpha/2}, \quad x \in \Rd.
		\]
		Using this together with \eqref{eq:21} and \eqref{eq:hk}, we infer that, for $x,y \in \Gamma$,
		\begin{align}\label{eq:26}\begin{aligned}
				\rho_1(x,y) &\approx p_1(x,y) \frac{\big(1 \wedge \delta_{\Gamma}(x)\big)^{\alpha/2} \big(1 \wedge \delta_{\Gamma}(y)\big)^{\alpha/2}}{\big(1 \wedge |x|)^{\alpha/2-\beta} \big(1 \wedge |y|)^{\alpha/2-\beta} \delta_{\Gamma}(x)^{\alpha/2}|x|^{\beta-\alpha/2}\delta_{\Gamma}(y)^{\alpha/2}|y|^{\beta-\alpha/2}} \\ &\approx \big(1+|x-y|\big)^{-d-\alpha} \frac{\big(1+\delta_{\Gamma}(x))^{-\alpha/2} \big(1+\delta_{\Gamma}(y))^{-\alpha/2}}{(1+|x|)^{\beta-\alpha/2} (1+|y|)^{\beta-\alpha/2}}.\end{aligned}
		\end{align}
	Let $q \in (1,\infty)$ and assume $\beta \geq \alpha/2$. 
	Then it follows from \eqref{eq:26} that $\rho_1(x,y) \lesssim 1$  for $x,y \in \Gamma$, and \eqref{eq:rho_density} entails that
	\begin{align}\label{eq:31}\begin{aligned}
		\int_{\Gamma} \rho_1(x,y)^q M_{\Gamma}^2(x)\ud x &\leq \| \rho_1(x,\,\cdot\,) \|_{\infty}^{q-1}\int_{\Gamma}\rho_1(x,y)M_{\Gamma}^2(y)\ud y \\ &= \| \rho_1(x,\,\cdot\,) \|_{\infty}^{q-1}\lesssim 1.\end{aligned}
	\end{align}
	Thus, we get \eqref{eq:20} as claimed.
	
	Now assume that $d \geq 2$ and $\beta <\alpha/2$. Let $y \in \Gamma$ be such that $\delta_{\Gamma}(y)=2$, so that $A:=B(y,1) \subseteq \Gamma$. Then for $x \in A$ one clearly has that $1+|x-y| \approx 1$ and $\delta_{\Gamma}(x) \approx 1$. Then it follows from \eqref{eq:26} that
	\begin{align*}
		\int_{\Gamma} \rho_1(x,y)^qM_{\Gamma}^2(x)\ud x &\geq \int_A \rho_1(x,y)^qM_{\Gamma}^2(x)\ud x \\ &\approx \int_A (1+|x|)^{q(\alpha/2-\beta)}(1+|y|)^{q(\alpha/2-\beta)} \delta_{\Gamma}(x)^{\alpha}|x|^{2\beta-\alpha}\ud x \\ &\approx |y|^{(q-1)(\alpha-2\beta)}.
	\end{align*}
	Since $\alpha-2\beta>0$ and $q>1$, by taking $|y| \to \infty$ we see that \eqref{eq:20} cannot hold in this case.
\end{proof}
\begin{corollary}\label{cor:s_cones}
	For $d \geq 2$ and smooth cone $\Gamma$, \eqref{eq:20} holds if and only if $\beta \geq \alpha/2$. 
\end{corollary}
\begin{proof}
	Recall that $\Gamma$ is open and $C^{1,1}$ outside of the origin.
	From the harmonicity and homogeneity of $M_{\Gamma}$, by the boundary Harnack principle we get, as in \cite[Lemma 3.3]{Michalik06}, that 
	\[
	M_{\Gamma}(x) \approx \delta_{\Gamma}(x)^{\alpha/2}|x|^{\beta-\alpha/2}, \quad x \in \Gamma.
	\]
	Moreover, since a smooth cone is fat, its Dirichlet heat kernel satisfies \eqref{eq:21}. Thus, one can directly repeat the proof of Lemma \ref{lem:rc_cones} to conclude the claim.
\end{proof}

\begin{example}\label{ex:1}
For $d=1$, either $\Gamma=(0,\infty)$ or $\Gamma = \R \setminus \{0\}$ and both cases are (trivially) smooth cones, with $\delta_{\Gamma}(y)=|y|$ for $y \in \Gamma$. 
	
	When $\Gamma = (0,\infty)$, then \eqref{eq:26} yields the boundedness of $\rho_1$ and \eqref{eq:20} holds through \eqref{eq:31}, since this $\Gamma$ is right-circular. Recall that here one has $\beta = \alpha/2$ and $M_\Gamma(x)=(0\vee x)^{\alpha/2}$ by \cite[Example 3.2]{BB04}, and we refer to Example~\ref{ex.pp} for the rest of the summary of this case.  
	
	If $\alpha \in (1,2)$ and $\Gamma=\R \setminus \{0\}$, then $\Gamma^c=\{0\}$ is a non-polar set, $\beta = \alpha-1$, and $M_\Gamma(x)=|x|^{\alpha-1}$; see \cite[Example 3.3]{BB04}. Note that in this example, $\Gamma$ is not right-circular anymore. Nevertheless, by \cite[Example 8]{KBTGMR10}, the survival probability is $\P_x(\tau_{\Gamma}>t) \approx (1 \wedge t^{-1/\alpha}|x|)^{\alpha-1}$ and
		\[
	p_1^{\Gamma}(x,y) \approx (1+|x-y|)^{-1-\alpha} (1 \wedge |x|)^{\alpha-1} (1 \wedge |y|)^{\alpha-1}, \quad x,y \in \Gamma.
	\]  
	Thus,
	\[
	\rho_1(x,y) \approx (1+|x-y|)^{-1-\alpha} (1+|x|)^{1-\alpha}(1+|y|)^{1-\alpha} \lesssim 1, \quad x,y \in \Gamma,
	\]
	and one may apply \eqref{eq:31} to get \eqref{eq:20}, too. It then follows from Theorem~\ref{thm_lim_norm}  with $q=2$ that $\lim_{t \to \infty} t^{(2\alpha-1)/(2\alpha)}\|P_t^{\Gamma}f\|_{2}=0$ if $\int_\R f(x)|x|^{\alpha-1}\ud x=0$. Accordingly, by Lemma~\ref{lem:vphi_redef}, the stationary density $\vphi$ of Theorem \ref{thm:stat_density} satisfies $\vphi(x) \approx (1+|x|)^{-2\alpha}$. Here and below, $x \in \Gamma$ and $t>0$. Then, $\Psi_t(x)\approx t^{-1} (1+t^{-1/\alpha}|x|)^{-1-\alpha}(1\wedge t^{-1/\alpha}|x|)^{\alpha-1}$ and the density of the Yaglom limit is comparable with $\vphi(x)|x|^{\alpha-1} \approx  (1+|x|)^{-2\alpha} |x|^{\alpha-1} \approx \Psi_1(x)$.

%
%

\end{example}
%
%

\begin{example}\label{ex:rc_cones}
	Let $d \geq 2$ and $\Gamma$ be a right-circular cone which is a subset of the half-space $\mathbb{R}_+^d$. Then by \cite[Example 3.2 and Lemma 3.3]{BB04}, we have $\beta \geq \alpha/2$ and Lemma \ref{lem:rc_cones} gives \eqref{eq:20}. On the other hand, if $\Gamma$ is such that $\mathbb{R}_+^d \subsetneq \Gamma$, then $\beta < \alpha/2$ by \cite[Lemma 3.3]{BB04}, and Lemma \ref{lem:rc_cones} asserts that \eqref{eq:20} does not hold.
	
The following extends Example~\ref{ex.pp}.	Let $\Gamma=\mathbb{R}_+^d$, so that $M_{\Gamma}(x) = (0\vee x_d)^{\alpha/2}$. By Theorem \ref{thm_lim_norm}, $\lim_{t \to \infty} t^{(d+\alpha)/(2\alpha)} \|P_t^{\Gamma}f\|_2=0$ if $\int_{\R^d_+} f(x)x_d^{\alpha/2}\ud x=0$. Furthermore, the survival probability is $\P_x(\tau_{\Gamma}>t) \approx 1 \wedge t^{-1/2} (0\vee x_d)^{\alpha/2}$ and 
	\[
	p_1^{\Gamma}(x,y) \approx (1 + |x-y|)^{-d-\alpha} (1 \wedge x_d)^{\alpha/2}(1 \wedge y_d)^{\alpha/2}, 
	\]
	see \cite[Example 2]{KBTGMR10}. Here and below, $x,y \in \Gamma$. So, $\Psi_t(x) \approx (t^{1/\alpha} \vee |x|)^{-d-\alpha}(t^{1/\alpha} \wedge x_d)^{\alpha/2}$, $t>0$. The stationary density $\vphi$ is  comparable to $(1+|x|)^{-d-\alpha} (1+x_d)^{-\alpha/2}$ and the density of the Yaglom distribution is $\vphi(x)x_d^{\alpha/2}/\int_{\R^d_+} \vphi(y)y_d^{\alpha/2}\ud y \approx (1+|x|)^{-d-\alpha} (1\wedge x_d)^{\alpha/2} \approx \Psi_1(x)$. 
\end{example}

	\bibliographystyle{abbrv}
	\bibliography{bib-file}

\begin{thebibliography}{10}

\bibitem{ABGLW21}
G.~Armstrong, K.~Bogdan, T.~Grzywny, {\L{}}.~Le\.{z}aj, and L.~Wang.
\newblock Yaglom limit for unimodal {L}\'{e}vy processes.
\newblock {\em Ann. Inst. Henri Poincar\'{e} Probab. Stat.}, 59(3):1688--1721,
  2023.

\bibitem{MR1042064}
R.~Ba{\~n}uelos, R.~F. Bass, and K.~Burdzy.
\newblock A representation of local time for {L}ipschitz surfaces.
\newblock {\em Probab. Theory Related Fields}, 84(4):521--547, 1990.

\bibitem{BB04}
R.~Ba{\~n}uelos and K.~Bogdan.
\newblock Symmetric stable processes in cones.
\newblock {\em Potential Anal.}, 21(3):263--288, 2004.

\bibitem{MR4490672}
D.~Berger, F.~K\"{u}hn, and R.~L. Schilling.
\newblock L\'{e}vy processes, generalized moments and uniform integrability.
\newblock {\em Probab. Math. Statist.}, 42(1):109--131, 2022.

\bibitem{BILER2001613}
P.~Biler, G.~Karch, and W.~A. Woyczyński.
\newblock Critical nonlinearity exponent and self-similar asymptotics for lévy
  conservation laws.
\newblock {\em Annales de l'Institut Henri Poincaré C, Analyse non linéaire},
  18(5):613--637, 2001.

\bibitem{RB92}
R.~M. Blumenthal.
\newblock {\em Excursions of {M}arkov processes}.
\newblock Probability and its Applications. Birkh\"auser Boston, Inc., Boston,
  MA, 1992.

\bibitem{MR0119247}
R.~M. Blumenthal and R.~K. Getoor.
\newblock Some theorems on stable processes.
\newblock {\em Trans. Amer. Math. Soc.}, 95:263--273, 1960.

\bibitem{Bogachev07}
V.~I. Bogachev.
\newblock {\em Measure theory. {V}ol. {I}, {II}}.
\newblock Springer-Verlag, Berlin, 2007.

\bibitem{KBTG10}
K.~Bogdan and T.~Grzywny.
\newblock Heat kernel of fractional {L}aplacian in cones.
\newblock {\em Colloq. Math.}, 118(2):365--377, 2010.

\bibitem{KBTGMR10}
K.~Bogdan, T.~Grzywny, and M.~Ryznar.
\newblock Heat kernel estimates for the fractional {L}aplacian with {D}irichlet
  conditions.
\newblock {\em Ann. Probab.}, 38(5):1901--1923, 2010.

\bibitem{KBTGMR-dhk}
K.~Bogdan, T.~Grzywny, and M.~Ryznar.
\newblock Dirichlet heat kernel for unimodal {L}\'evy processes.
\newblock {\em Stochastic Process. Appl.}, 124(11):3612--3650, 2014.

\bibitem{MR2182071}
K.~Bogdan and T.~Jakubowski.
\newblock Probl\`eme de {D}irichlet pour les fonctions {$\alpha$}-harmoniques
  sur les domaines coniques.
\newblock {\em Ann. Math. Blaise Pascal}, 12(2):297--308, 2005.

\bibitem{BJKP22}
K.~Bogdan, T.~Jakubowski, P.~Kim, and D.~Pilarczyk.
\newblock Self-similar solutions for {H}ardy operator.
\newblock {\em J. Funct. Anal.}, 285(5):110014, 2023.

\bibitem{bogdan2023shotdown}
K.~Bogdan, K.~Jastrz\k{e}bski, M.~Kassmann, M.~Kijaczko, and P.~Pop\l{}awski.
\newblock Shot-down stable processes.
\newblock Preprint, arXiv:2301.12290, 2023.

\bibitem{BoPaWa2018}
K.~Bogdan, Z.~Palmowski, and L.~Wang.
\newblock Yaglom limit for stable processes in cones.
\newblock {\em Electron. J. Probab.}, 23:1--19, 2018.

\bibitem{MR3737628}
K.~Bogdan, J.~Rosi\'{n}ski, G.~Serafin, and {\L}.~Wojciechowski.
\newblock L\'{e}vy systems and moment formulas for mixed {P}oisson integrals.
\newblock In {\em Stochastic analysis and related topics}, volume~72 of {\em
  Progr. Probab.}, pages 139--164. Birkh\"{a}user/Springer, Cham, 2017.

\bibitem{MR3339224}
K.~Bogdan, B.~Siudeja, and A.~St{\'o}s.
\newblock Martin kernel for fractional {L}aplacian in narrow cones.
\newblock {\em Potential Anal.}, 42(4):839--859, 2015.

\bibitem{Brezis11}
H.~Brezis.
\newblock {\em Functional analysis, {S}obolev spaces and partial differential
  equations}.
\newblock Universitext. Springer, New York, 2011.

\bibitem{MR4546021}
N.~Champagnat and D.~Villemonais.
\newblock General criteria for the study of quasi-stationarity.
\newblock {\em Electron. J. Probab.}, 28:Paper No. 22, 84, 2023.

\bibitem{MR3160562}
L.~Chaumont, H.~Pant\'{\i}, and V.~Rivero.
\newblock The {L}amperti representation of real-valued self-similar {M}arkov
  processes.
\newblock {\em Bernoulli}, 19(5B):2494--2523, 2013.

\bibitem{MR4596024}
J.~W. Cholewa and A.~Rodriguez-Bernal.
\newblock Self-similarity in homogeneous stationary and evolution problems.
\newblock {\em J. Evol. Equ.}, 23(2):42, 2023.

\bibitem{MR1329992}
K.~L. Chung and Z.~X. Zhao.
\newblock {\em From {B}rownian motion to {S}chr\"odinger's equation}.
\newblock Springer-Verlag, Berlin, 1995.

\bibitem{HassRivero12}
B.~Haas and V.~Rivero.
\newblock Quasi-stationary distributions and {Y}aglom limits of self-similar
  {M}arkov processes.
\newblock {\em Stochastic Process. Appl.}, 122(12):4054--4095, 2012.

\bibitem{Kallenberg02}
O.~Kallenberg.
\newblock {\em Foundations of modern probability}.
\newblock Probability and its Applications. Springer-Verlag, New York, 2002.

\bibitem{Kulczycki99}
T.~Kulczycki.
\newblock Exit time and {G}reen function of cone for symmetric stable
  processes.
\newblock {\em Probab. Math. Statist.}, 19(2, Acta Univ. Wratislav. No.
  2198):337--374, 1999.

\bibitem{KulikScheutzow15}
A.~Kulik and M.~Scheutzow.
\newblock A coupling approach to {D}oob's theorem.
\newblock {\em Atti Accad. Naz. Lincei Rend. Lincei Mat. Appl.}, 26(1):83--92,
  2015.

\bibitem{MR4155177}
A.~E. Kyprianou, V.~Rivero, B.~\c{S}eng\"{u}l, and T.~Yang.
\newblock Entrance laws at the origin of self-similar {M}arkov processes in
  high dimensions.
\newblock {\em Trans. Amer. Math. Soc.}, 373(9):6227--6299, 2020.

\bibitem{MR4415390}
A.~E. Kyprianou, V.~Rivero, and W.~Satitkanitkul.
\newblock Stable {L}\'{e}vy processes in a cone.
\newblock {\em Ann. Inst. Henri Poincar\'{e} Probab. Stat.}, 57(4):2066--2099,
  2021.

\bibitem{Michalik06}
K.~Michalik.
\newblock Sharp estimates of the {G}reen function, the {P}oisson kernel and the
  {M}artin kernel of cones for symmetric stable processes.
\newblock {\em Hiroshima Math. J.}, 36(1):1--21, 2006.

\bibitem{MR4320772}
P.~Patie and M.~Savov.
\newblock Spectral expansions of non-self-adjoint generalized {L}aguerre
  semigroups.
\newblock {\em Mem. Amer. Math. Soc.}, 272(1336):vii+182, 2021.

\bibitem{Pruitt81}
W.~E. Pruitt.
\newblock The growth of random walks and {L}\'evy processes.
\newblock {\em Ann. Probab.}, 9(6):948--956, 1981.

\bibitem{Rudin91}
W.~Rudin.
\newblock {\em Functional analysis}.
\newblock International Series in Pure and Applied Mathematics. McGraw-Hill,
  Inc., New York, second edition, 1991.

\bibitem{MR3644418}
R.~L. Schilling.
\newblock {\em Measures, integrals and martingales}.
\newblock Cambridge University Press, Cambridge, second edition, 2017.

\bibitem{MR1719233}
R.~Song and J.-M. Wu.
\newblock Boundary {H}arnack principle for symmetric stable processes.
\newblock {\em J. Funct. Anal.}, 168(2):403--427, 1999.

\bibitem{MR3063313}
E.~A. van Doorn and P.~K. Pollett.
\newblock Quasi-stationary distributions for discrete-state models.
\newblock {\em European J. Oper. Res.}, 230(1):1--14, 2013.

\bibitem{Vasquez18}
J.~L. V\'{a}zquez.
\newblock Asymptotic behaviour for the fractional heat equation in the
  {E}uclidean space.
\newblock {\em Complex Var. Elliptic Equ.}, 63(7-8):1216--1231, 2018.

\bibitem{VAZQUEZ2020112034}
J.~L. V\'{a}zquez.
\newblock The evolution fractional p-{L}aplacian equation in $\mathbb{R}^n$.
  fundamental solution and asymptotic behaviour.
\newblock {\em Nonlinear Analysis}, 199:112034, 2020.

\end{thebibliography}
	
\end{document}